\documentclass{daj}
\usepackage{amsthm,amsmath,amssymb,youngtab,verbatim,tikz,amsmath,mathtools,mathrsfs}
\theoremstyle{definition}
\newtheorem{defn}{Definition}
\newtheorem{remark}{Remark}

\newtheorem{claim}{Claim}
\theoremstyle{theorem}
\newtheorem{theorem}{Theorem}
\newtheorem{lem}[theorem]{Lemma}

\newtheorem{prop}[theorem]{Proposition}

\dajAUTHORdetails{%
  title = {Forbidden Intersection Problems for Families of Linear Maps}, 
  author = {David Ellis, Guy Kindler, and Noam Lifshitz},
  plaintextauthor = {David Ellis, Guy Kindler, and Noam Lifshitz},
  runningtitle = {Forbidden Intersection Problems for Families of Linear Maps}, 
}   

\dajEDITORdetails{%
   year={2023},
   number={19},
   received={13 September 2022},   
   published={12 December 2023},  
   doi={10.19086/da.90718},       
}   

\begin{document}
\newcommand{\f}{\mathcal{F}}
\newcommand{\g}{\mathcal{G}}
\newcommand{\J}{\mathcal{J}}
\newcommand{\Cay}{\mathrm{Cay}}
\newcommand{\Trace}{\mathrm{Trace}}
\newcommand{\Image}{\mathrm{Im}}
\newcommand{\GL}{\mathrm{GL}}
\newcommand{\SL}{\mathrm{SL}}
\newcommand{\Span}{\mathrm{Span}}
\newcommand{\Domain}{\mathrm{Domain}}
\newcommand{\codim}{\mathrm{codim}}
\newcommand{\rank}{\mathrm{rank}}
\newcommand{\Lin}{\mathcal{L}}
\newcommand{\h}{\mathcal{H}}
\newcommand{\tree}{\mathcal{T}}
\newcommand{\A}{\mathcal{A}}

\begin{frontmatter}[classification=text]


\author[david]{David Ellis}
\author[guy]{Guy Kindler}
\author[noam]{Noam Lifshitz\thanks{Supported in part by ISF grant 1980/22.}}

\begin{abstract}
We study an analogue of the Erd\H{o}s-S\'os forbidden intersection problem, for families of invertible linear maps. If $V$ and $W$ are vector spaces over the same field, we say a family $\mathcal{F}$ of linear maps from $V$ to $W$ is {\em $(t-1)$-intersection-free} if for any two linear maps $\sigma_1,\sigma_2 \in \mathcal{F}$, the dimension of the subspace $\{v \in V:\ \sigma_1(v)=\sigma_2(v)\}$ is not equal to $t-1$. We prove that if $n$ is sufficiently large depending on $t$, $q$ is any prime power, $V$ is an $n$-dimensional vector space over $\mathbb{F}_q$, and $\mathcal{F} \subset \textrm{GL}(V)$ is $(t-1)$-intersection-free, then
$$|\f| \leq \prod_{i=1}^{n-t}(q^n - q^{i+t-1}).$$
Equality holds only if there exists a $t$-dimensional subspace of $V$ on which all elements of $\mathcal{F}$ agree, or a $t$-dimensional subspace of $V^*$ on which all elements of $\{\sigma^*:\ \sigma \in \mathcal{F}\}$ agree.

Our main tool is a `junta approximation' result for families of linear maps with a forbidden intersection: namely, that if $V$ and $W$ are finite-dimensional vector spaces over the same finite field, then any $(t-1)$-intersection-free family of linear maps from $V$ to $W$ is essentially contained in a $t$-intersecting {\em junta} (meaning, a family $\mathcal{J}$ of linear maps from $V$ to $W$ such that the membership of $\sigma$ in $\mathcal{J}$ is determined by $\sigma(v_1),\ldots,\sigma(v_M),\sigma^*(a_1),\ldots,\sigma^*(a_N)$, where $v_1,\ldots,v_M \in V$, $a_1,\ldots,a_N \in W^*$ and $M+N$ is bounded). The proof of this in turn relies on a variant of the `junta method' (originally introduced by Dinur and Friedgut \cite{dinur-friedgut} and powerfully extended by Keller and the last author \cite{keller-lifshitz}), together with spectral techniques and a new hypercontractive inequality.
\end{abstract}
\end{frontmatter}


\section{Introduction}

Erd\H{o}s-Ko-Rado type problems are an important class of problems within extremal combinatorics. In general, an Erd\H{o}s-Ko-Rado type problem asks for the maximum possible size of a family of objects, subject to some intersection condition on pairs of objects in the family. For example, we say a family of sets is {\em intersecting} if any two sets in the family have nonempty intersection. The classical Erd\H{o}s-Ko-Rado theorem \cite{ekr} states that if $k < n/2$, an intersecting family of $k$-element subsets of $\{1,2,\ldots,n\}$ has size at most ${n-1 \choose k-1}$, and that if equality holds, then the family must consist of all $k$-element subsets containing some fixed element. Over the last sixty years, many other Erd\H{o}s-Ko-Rado type results have been obtained, for different mathematical structures (e.g.\ for families of graphs \cite{cfgs,triangle}, and families of partitions \cite{meagher}) and under different intersection conditions on the sets in the family. We mention in particular the seminal theorem of Ahlswede and Khachatrian \cite{ak} which specifies, for each $(n,k,t) \in \mathbb{N}^3$, the largest possible size of a $t$-intersecting family of $k$-element subsets of $\{1,2,\ldots,n\}$. (We say a family of sets is {\em $t$-intersecting} if any two sets in the family have intersection of size at least $t$.) As well as being natural in their own right, Erd\H{o}s-Ko-Rado type questions have found applications e.g.\ in computer science and coding theory, and the techniques developed in solving them have found wide applicability in many other areas of Mathematics. The reader is referred to \cite{DF83,MV15} for surveys of this area of research and its applications.

One particularly challenging type of Erd\H{o}s-Ko-Rado problem concerns what happens when just one intersection-size is forbidden. The {\em Erd\H{o}s-S\'os forbidden intersection problem} \cite{erdos-sos} is to determine, for each $(n,k,t) \in \mathbb{N}^3$, the maximum possible size of a family of $k$-element subsets of $\{1,2,\ldots,n\}$ such that no two sets in the family have intersection of size exactly $t-1$. This problem remains open in full generality, unlike the $t$-intersection problem above (solved by Ahlswede and Khachatrian), though it has been solved for quite a wide range of the parameters, by Frankl and F\"uredi \cite{ff}, Keevash, Mubayi and Wilson \cite{kmw}, Keller and the authors \cite{ekl}, Keller and the second author \cite{keller-lifshitz} and Kupavskii and Zakharov \cite{kz}. These solutions have involved a wide range of methods (combinatorial, probabilistic, algebraic and Fourier-analytic), some of which have found important applications elsewhere. For example, the work of Frankl and F\"uredi \cite{ff} was one of the first uses of their widely-applicable `delta-system method', and the work of Keller and the last author \cite{keller-lifshitz} involved a broad extension of the `junta method', which has also been widely used.

In this paper, we obtain forbidden intersection theorems for families of linear maps, and for subsets of finite general linear groups. If $V$ and $W$ are vector spaces over the same field, we let $\mathcal{L}(V,W)$ denote the space of linear maps from $V$ to $W$, and we let $\GL(V)$ denote the group of all invertible linear maps from $V$ to itself. If $\mathcal{F} \subset \mathcal{L}(V,W)$ and $t \in \mathbb{N}$, we say $\mathcal{F}$ is {\em $t$-intersecting} if 
$$\dim(\{v \in V:\ \sigma_1(v) = \sigma_2(v)\}) \geq t$$
for all $\sigma_1,\sigma_2 \in \mathcal{F}$, and we say that $\mathcal{F}$ is {\em $(t-1)$-intersection-free} if 
$$\dim(\{v \in V:\ \sigma_1(v) = \sigma_2(v)\}) \neq t-1$$
for all $\sigma_1,\sigma_2 \in \mathcal{F}$. Clearly, a $t$-intersecting family is $(t-1)$-intersection free.

One of our main extremal results, concerning sufamilies of $\GL(\mathbb{F}_q^n)$ that are $(t-1)$-intersection-free, is as follows.

\begin{theorem}
\label{theorem:glnfq}
For any $t \in \mathbb{N}$, there exists $n_0 = n_0(t) \in \mathbb{N}$ such that the following holds. If $n \in \mathbb{N}$ with $n \geq n_0$, $q$ is a prime power, $V$ is an $n$-dimensional vector space over $\mathbb{F}_q$, and $\f \subset \GL(V)$ is $(t-1)$-intersection-free, then
$$|\f| \leq \prod_{i=1}^{n-t}(q^n - q^{i+t-1}).$$
Equality holds only if there exists a $t$-dimensional subspace $U$ of $V$ on which all elements of $\f$ agree, or a $t$-dimensional subspace $A$ of $V^*$ on which all elements of $\{\sigma^*:\ \sigma \in \f\}$ agree.
\end{theorem}

We note that, in recent independent work, Ernst and Schmidt \cite{es} obtained the upper bound in this theorem under a somewhat stronger hypothesis, namely, that the family $\mathcal{F}$ is $t$-intersecting, using an algebraic (spectral) proof. They also showed that for any $t$-intersecting family $\mathcal{F}$ attaining equality, the characteristic function of $\mathcal{F}$ lies in the linear span of the characteristic functions of the families in the equality part of Theorem \ref{theorem:glnfq}, though they were not able to give a combinatorial characterisation of the families attaining equality. Our techniques are different to those of Ernst and Schmidt: we employ probabilistic and combinatorial techniques to reduce to the case of `highly quasirandom' families of linear maps, which can then be dealt with by a (much simpler) spectral argument. We note that Ernst and Schmidt conjecture in \cite{es} that Theorem \ref{theorem:glnfq} holds. We note also that, as in the result of Ernst and Schmidt, the $n_0$ in Theorem \ref{theorem:glnfq} is independent of $q$.

We note also that the $t=1$ case of the upper bound in Theorem \ref{theorem:glnfq} is identical to the $t=1$ case of the theorem of Ernst and Schmidt. In fact, as Ernst and Schmidt mention, the $t=1$ case of their theorem was already known; indeed, it follows immediately from the (classical) fact that $\GL(\mathbb{F}_q^n)$ contains a Singer cycle; see for example \cite{aa} or \cite{am}. (Here is a proof-sketch: identify $\mathbb{F}_q^n$ with $\mathbb{F}_{q^n}$, and let $x$ be a generator of the cyclic group $\mathbb{F}_{q^n}^{\times}$; the subgroup $H = \langle x \rangle$ has the property that $\sigma_1(v) \neq \sigma_2(v)$ for all distinct $\sigma_1,\sigma_2 \in H$ and all $v \neq 0$, and so does any left coset of $H$, so each left coset of $H$ contains at most one element of a 1-intersecting family $\mathcal{F}$. Hence, $|\mathcal{F}| \leq \# \text{ of left cosets of }H = |\GL(\mathbb{F}_q^n)|/(q^n-1)$, which is the desired bound. Interestingly, no such simple proof of the theorem of Ernst and Schmidt is known for any $t \geq 2$.)

Our main tool for proving Theorem \ref{theorem:glnfq} is the following `junta approximation' theorem for families of linear maps with a forbidden intersection.

\begin{theorem}
\label{theorem:junta-approximation}
For any $t \leq r \in \mathbb{N}$, there exists $m_0 = m_0(r,t) \in \mathbb{N}$ such that the following holds. If $m,n \in \mathbb{N}$ with $m_0 \leq m \leq n \leq (1+1/(9t))m$, $V$ is an $m$-dimensional vector space over $\mathbb{F}_q$, $W$ is an $n$-dimensional vector space over $\mathbb{F}_q$, and $\f \subset \Lin(V,W)$ is $(t-1)$-intersection-free, then there exists a strongly $t$-intersecting $(C,r)$-junta $\J \subset \Lin(V,W)$ such that $|\f \setminus \J| \leq C|\Lin(V,W)|/q^{rm}$, where $C = q^{2146r^4}$.
\end{theorem}

The reader is referred to Definitions \ref{defn:junta} and \ref{defn:strongly-intersecting} below, for the formal definitions of a $(C,r)$-junta and of being strongly $t$-intersecting. Suffice it to say at this stage that if $\mathcal{J}$ is a $(C,r)$-junta, then the membership of $\sigma$ in $\mathcal{J}$ is determined by
$$\sigma(v_1),\ldots,\sigma(v_M),\sigma^*(a_1),\ldots,\sigma^*(a_N),$$
where $v_1,\ldots,v_M \in V$, $a_1,\ldots,a_N \in W^*$ and $M+N \leq Cr$. (The property of being `strongly' $t$-intersecting is at first sight a little stronger, but in fact equivalent, to being $t$-intersecting, though we do not need the equivalence.) Informally, Theorem \ref{theorem:junta-approximation} says that a $(t-1)$-intersection-free family is almost contained within a $t$-intersecting junta of bounded complexity.

The proof of Theorem \ref{theorem:junta-approximation} follows a similar structure to the proof of the analogous result for families of permutations, due to the first and last authors \cite{el}. In a similar way to in \cite{el}, we first obtain a weak regularity lemma (Lemma \ref{lem:reg-lem}) for subsets of $\Lin(V,W)$; for any family $\f \subset \Lin(V,W)$, this lemma outputs a junta $\mathcal{J}$ such that (i) $\J$ almost contains $\f$ and (ii) the constituent parts of $\J$ intersect $\f$ in a weakly pseudorandom way (this weak notion of psuedorandomness is termed `uncaptureability'). The rest of the proof consists of showing that if $\f$ is $(t-1)$-intersection-free, then the junta $\J$ is (strongly) $t$-intersecting. We accomplish this in several steps. First, we use a combinatorial argument to `bootstrap' the weak notion of pseudorandomness into a stronger one (which we term `quasiregularity'), possibly at the cost of passing to reasonably dense subsets. We then show that the quasiregularity of a set implies a rather strong bound on the $L^2$-norm of the low-degree part of its indicator function; this step relies on a new hypercontractive inequality (Proposition \ref{prop:higher-norms}). A (relatively short) spectral argument then completes the proof; the above-mentioned $L^2$-bound ensures that relatively crude eigenvalue estimates suffice.

We note that the proof in \cite{el} does not require any use of hypercontractivity, due to the very close relationship between `combinatorial' and `algebraic' quasirandomness for subsets of the symmetric group, exploited in \cite{el}. The relationship is not quite so close in the setting of $\mathcal{L}(V,W)$, necessitating the use of hypercontractivity.
 
\subsection*{Some notation and definitions}

Throughout, $q$ will be a prime power, and $V$ and $W$ will be vector spaces over $\mathbb{F}_q$ with $\dim(V)=m$ and $\dim(W)=n$. We write $\mathcal{M}(n,m)$ for the set of all $n$ by $m$ matrices\footnote{Meaning, as usual, $n$ rows and $m$ columns.} with entries in $\mathbb{F}_q$; of course, fixing bases of $V$ and $W$ yields a one-to-one correspondence between $\Lin(V,W)$ and $\mathcal{M}(n,m)$. As usual, if $U$ is a vector space over $\mathbb{F}_q$, we denote by $U^*$ its dual, i.e., the vector space $\Lin(U,\mathbb{F}_q)$ of linear functionals.

We denote linear maps from subspaces of $V$ to subspaces of $W$ by upper-case Greek letters, e.g.\ $\Pi$, and linear maps from subspaces of $W^*$ to subspaces of $V^*$ by lower-case Greek letters, e.g.\ $\pi$. If $\Pi_i \in \Lin(S_i,W)$ for each $i \in \{1,2\}$, where $S_i \leq V$ for each $i \in \{1,2\}$, we write
$$\mathfrak{a}(\Pi_1,\Pi_2) := \{v \in S_1 \cap S_2:\ \Pi_1(v) = \Pi_2(v)\}$$
for the subspace of $V$ (and of $S_1 \cap S_2$) where $\Pi_1$ and $\Pi_2$ agree pointwise (and are defined). Similarly, if $\pi_i \in \Lin(A_i,V^*)$ for each $i \in \{1,2\}$, where $A_i \leq W^*$ for each $i \in \{1,2\}$, we write
$$\mathfrak{a}(\pi_1,\pi_2) := \{a \in A_1 \cap A_2:\ \pi_1(a) = \pi_2(a)\}$$
for the subspace of $W^*$ (and of $A_1 \cap A_2$) where $\pi_1$ and $\pi_2$ agree pointwise (and are defined). If $S_1 \cap S_2 = \{0\}$, we define $\Pi_1+\Pi_2 \in \Lin(S_1 \oplus S_2,W)$ in the natural way, i.e.\ by $(\Pi_1+\Pi_2)(s_1 + s_2) = \Pi_1(s_1)+\Pi_2(s_2)$ for all $s_1 \in S_1$ and $s_2 \in S_2$. If $A_1 \cap A_2 = \{0\}$, we define $\pi_1+\pi_2$ analogously.

If $S_{1},S_{2} \leq V$, $A_1,A_2 \leq W^{*}$, $\Pi_{i} \in \Lin(S_i,W)$ for each $i \in \{1,2\}$, $\pi_i \in \Lin(A_i,V^*)$ for each $i \in \{1,2\}$, and $\f \subset \Lin(V,W)$, we write $\f(\Pi_1,\pi_1,\overline{\Pi_2},\overline{\pi_2})$ for the set of all linear maps $\sigma \in \f$ such that $\sigma$ agrees with $\Pi_{1}$ on every element of $S_{1}$, $\sigma^*$ agrees with $\pi_1$ on every element of $A_1$, $\sigma$ disagrees with $\Pi_{2}$ on every non-zero element of $S_{2}$, and $\sigma^*$ disagrees with $\pi_2$ on every non-zero element of $A_2$. It will be convenient for us to regard $\f(\Pi_1,\pi_1,\overline{\Pi_2},\overline{\pi_2})$ as a subset of $\Lin(V,W)(\Pi_1,\pi_1)$, equipping the latter with the uniform measure (which we denote by $\mu^{|\,\Pi_1,\pi_1}$), so that
\begin{align*} \mu^{|\,\Pi_1,\pi_1}(\f(\Pi_1,\pi_1,\overline{\Pi_2},\overline{\pi_2})) & := \frac{|\f(\Pi_1,\pi_1,\overline{\Pi_2},\overline{\pi_2})|}{|\Lin(V,W)(\Pi_1,\pi_1)|}\\
&= q^{-(m-\dim(\Domain(\Pi_1)))(n-\dim(\Domain(\pi_1)))} |\f(\Pi_1,\pi_1,\overline{\Pi_2},\overline{\pi_2})|.\end{align*}
Similarly, if $S \leq V$, $A \leq W^{*}$, $\Pi \in \Lin(S,W)$, $\pi \in \Lin(A,V^*)$ and $f:\Lin(V,W) \to \mathbb{R}_{\geq 0}$, we write $f(\Pi,\pi)$ for the restriction of $f$ to $\Lin(V,W)(\Pi,\pi)$, and we again equip $\Lin(V,W)(\Pi,\pi)$ with the uniform measure $\mu^{|\,\Pi,\pi}$, writing
$$\mathbb{E}^{|\,\Pi,\pi}[f] := q^{-(m-\dim(\Domain(\Pi)))(n-\dim(\Domain(\pi)))}\sum_{\sigma \in \Lin(V,W)(\Pi,\pi)} f(\sigma).$$

Sometimes, abusing notation slightly, we will write $\mu^{|\cdots}(\mathcal{F}(\Pi,\pi,\overline{\Sigma},\overline{\sigma}))$ in place of $\mu^{|\,\Pi,\pi}(\f(\Pi,\pi,\overline{\Sigma}, \overline{\sigma}))$, etc.; in such cases the reader should mentally replace the dots in the superscript by the linear maps that appear in brackets after $\f$, without bars. (We reassure the reader that this notation will only be used when the measure in question is clear from the context, i.e., from the way the family $\f(\cdots)$ is written, using the bracket notation.)

We now give our (precise) definition of a `junta', for families of linear maps.
\begin{defn}
\label{defn:junta} For $C,r \in \mathbb{N}$, we say a family $\J \subset \Lin(V,W)$ is a {\em $(C,r)$-junta} if there exist subspaces $S_1,S_2,\ldots,S_N \leq V$, subspaces $A_1,\ldots,A_N \leq W^*$, linear maps $\Pi_i : S_i \to W$ ($1 \leq i \leq N$), and linear maps $\pi_i: A_i \to V^*$ ($1 \leq i \leq N$), such that $N \leq C$, $\dim(S_i)+\dim(A_i) \leq r$ for all $i$, and $\J$ consists of all linear maps $\sigma \in \Lin(V,W)$ such that for some $i \in [N]$, $\sigma$ agrees with $\Pi_i$ on every element of $S_i$ and $\sigma^*$ agrees with $\pi_i$ on every element of $A_i$. In this case, we write $\J = \langle (\Pi_1,\pi_1),(\Pi_2,\pi_2),\ldots(\Pi_N,\pi_N)\rangle$.
\end{defn}

\begin{defn}
\label{defn:strongly-intersecting}
    We say that a $(C,r)$-junta $\J \subset \Lin(V,W)$ is {\em strongly $t$-intersecting} if we may write it in the above form,  $\J = \langle (\Pi_1,\pi_1),(\Pi_2,\pi_2),\ldots(\Pi_N,\pi_N)\rangle$, where for all $i,j \in [N]$, we have either $\dim(\mathfrak{a}(\Pi_i,\Pi_j)) \geq t$ or $\dim(\mathfrak{a}(\pi_i,\pi_j)) \geq t$.
\end{defn}

\begin{remark}
    Trivially, a strongly $t$-intersecting junta is $t$-intersecting. In fact, the converse holds; this is straightforward, if a little tedious, to check, and we omit the details as we do not in fact need the converse. It follows for $n$ sufficiently large depending on $t$, from our proof of Theorem \ref{theorem:junta-approximation} below.
\end{remark}

Our `weak' notion of pseudorandomness, for families of linear maps, is {\em uncaptureability}, defined as follows.

\begin{defn}
A family $\f\subset \Lin(V,W)$ is said to be \emph{$\left(s,\epsilon\right)$-captureable}
if there exist subspaces $S \leq V$ and $A \leq W^*$ with $\dim(S) + \dim(A) \leq s$, and linear maps $\Pi \in \Lin(S,W)$, $\pi \in \Lin(A,V^*)$, such that $\mu(\f(\overline{\Pi},\overline{\pi}))\leq \epsilon$. Similarly, if $\Pi_1:S_1\to W$ is a linear map with $S_1\leq V$, and $\pi_1:A_1 \to V^*$ is a linear map with $A_1 \leq W^*$, we say that a family $\f \subset \Lin(V,W)(\Pi_1,\pi_1)$ is {\em $\left(s,\epsilon\right)$-captureable} if there exist subspaces $S \leq V$ and $A \leq W^*$ with $S \cap S_1 = \{0\}$, $A \cap A_1 = \{0\}$ and $\dim(S) + \dim(A) \leq s$, and linear maps $\Pi \in \Lin(S,W)$, $\pi \in \Lin(A,V^*)$ such that $\mu^{|\,\Pi_1,\pi_1}\left(\f\left(\Pi_1,\pi_1,\overline{\Pi},\overline{\pi}\right)\right)\le\epsilon$. If a family is not $\left(s,\epsilon\right)$-captureable, then we say it is $\left(s,\epsilon\right)$-\emph{uncaptureable}.
\end{defn}

Slightly less formally, a family of linear maps $\f \subset \Lin(V,W)$ is highly uncaptureable if for any linear maps $\Pi$ and $\pi$ on bounded-dimensional subspaces (more precisely, subspaces of $V$ and of $W^*$, respectively), a significant fraction of the linear maps $\sigma \in \f$ are such that $\sigma$ disagrees everywhere with $\Pi$ (everywhere, that is, where $\Pi$ is defined) and $\sigma^*$ disagrees everywhere with $\pi$ (i.e.\ everywhere $\pi$ is defined). 

Our stronger notion of pseudorandomness, {\em quasiregularity}, is defined as follows.

\begin{defn}
For $s \in \mathbb{N}$ and $\alpha >0$, we say that a family $\f\subset \Lin(V,W)$ is \emph{$\left(s,\alpha\right)$-quasiregular}
if for any subspaces $S \leq V$ and $A \leq W^*$ with $\dim(S) + \dim(A) \leq s$, and any linear maps $\Pi \in \Lin(S,W)$, $\pi \in \Lin(A,V^*)$, we have $\mu^{|\,\Pi,\pi}(\f(\Pi,\pi)) \leq \alpha \mu(\f)$. If $\Pi_1:S_1\to W$ is a linear map with $S_1\leq V$, and $\pi_1:A_1 \to V^*$ is a linear map with $A_1 \leq W^*$, we say that a family $\f \subset \Lin(V,W)(\Pi_1,\pi_1)$ is {\em $\left(s,\alpha\right)$-quasiregular} if for any subspaces $S \leq V$ and $A \leq W^*$ with $S \cap S_1 = \{0\}$, $A \cap A_1 = \{0\}$ and $\dim(S) + \dim(A) \leq s$, and linear maps $\Pi \in \Lin(S,W)$, $\pi \in \Lin(A,V^*)$, we have  $\mu^{|\,\Pi_1,\pi_1,\Pi,\pi}\left(\f\left(\Pi_1,\pi_1,\Pi,\pi\right)\right)\leq \alpha(\mu^{|\,\Pi_1,\pi_1}\f(\Pi_1,\pi_1))$. Similarly, a function $f:\Lin(V,W) \to \mathbb{R}_{\geq 0}$ is said to be \emph{$\left(s,\alpha\right)$-quasiregular} if for any subspaces $S \leq V$ and $A \leq W^*$ with $\dim(S) + \dim(A) \leq s$, and any linear maps $\Pi \in \Lin(S,W)$, $\pi \in \Lin(A,V^*)$, we have $\mathbb{E}^{|\,\Pi,\pi}[f(\Pi,\pi)] \leq \alpha \mathbb{E}[f]$. If $\Pi_1:S_1\to W$ is a linear map with $S_1\leq V$, and $\pi_1:A_1 \to V^*$ is a linear map with $A_1 \leq W^*$, we say that a function $f:\Lin(V,W)(\Pi,\pi) \to \mathbb{R}_{\geq 0}$ is {\em $(s,\alpha)$-quasiregular} if for any subspaces $S \leq V$ and $A \leq W^*$ with $S \cap S_1 = \{0\}$, $A \cap A_1 = \{0\}$ and $\dim(S) + \dim(A) \leq s$, and linear maps $\Pi \in \Lin(S,W)$, $\pi \in \Lin(A,V^*)$, we have  $\mathbb{E}^{|\,\Pi_1,\pi_1,\Pi,\pi}[f\left(\Pi_1,\pi_1,\Pi,\pi\right)]\leq \alpha\mathbb{E}^{|\,\Pi_1,\pi_1}[f(\Pi_1,\pi_1)]$.
\end{defn}

Informally, a family of linear maps is highly quasiregular if restricting it to any (bounded-complexity) junta yields no large density-increment.

The following straightforward claim shows that quasiregularity does indeed imply uncaptureability (with appropriate choices of the parameters), justifying our use of the terms `strong' and `weak' above.

\begin{claim}
\label{claim:quasi-uncap} 
Let $\beta >1$, let $b,N \in \mathbb{N}$ and let $\delta >0$. If $\h(\Pi,\pi)$ is $(1,\beta)$-quasiregular, where
$$\dim(\Domain(\Pi))+\dim(\Domain(\pi)) \leq b,$$
and $\mu^{|\,\Pi,\pi}(\h(\Pi,\pi)) \geq \delta$, then $\h(\Pi,\pi)$ is $(N,\tfrac{1}{2}\delta)$-uncaptureable, provided $\beta < q^{\min\{m,n\}-N-b}/2$.
\end{claim}
\begin{proof}
Write $\h' = \h(\Pi,\pi)$. Suppose for a contradiction that $\h'$ is $(N,\tfrac{1}{2}\delta)$-captureable. Then there exist subspaces $S \leq V$ and $A \leq W^*$ with $S \cap \Domain(\Pi) = \{0\}$, $A \cap \Domain(\pi) = \{0\}$, and linear maps $\Psi \in \Lin(S,W)$, $\psi \in \Lin(A,V^*)$, such that $\mu^{|\,\Pi,\pi}(\h'(\overline{\Psi},\overline{\psi})) \leq \tfrac{1}{2}\delta$. It follows by a union bound that
$$\mu^{|\,\Pi,\pi}(\h') \leq \mu^{|\,\Pi,\pi}(\h'(\overline{\Psi},\overline{\psi}))+\beta q^N \mu^{|\,\Pi,\pi}(\h') q^{-\min\{m,n\}+b} \leq \tfrac{1}{2} \delta + \beta q^{N-\min\{m,n\}+b} \mu^{|\,\Pi,\pi}(\h'),$$
and therefore
$$\mu^{|\,\Pi,\pi}(\h') \leq \tfrac{1}{2} \delta / (1-\beta q^{N-\min\{m,n\}+b}) < \delta,$$
a contradiction.
\end{proof}

\section{A weak regularity lemma for subsets of $\Lin(V,W)$}
Our weak regularity lemma for subsets of $\Lin(V,W)$ is a relatively straightforward analogue of the corresponding results for subsets of $S_n$, in \cite{el}, except that here it is necessary to keep closer track of the parameter-dependence.

 \begin{lem}[Weak regularity lemma for subsets of $\Lin(V,W)$]
\label{lem:reg-lem} For each $r,s \in \mathbb{N} \cup \{0\}$, there exists $C = C(q,r,s) \leq 2q^r(q^{s}-1)^r$ such that for any family $\f\subset \Lin(V,W)$, there exists a $(C,r)$-junta
$$\J = \langle (\Pi_1,\phi_1),(\Pi_2,\phi_2),\ldots(\Pi_N,\phi_N)\rangle \subset \Lin(V,W)$$
such that
\begin{enumerate}
\item $\mu(\f\setminus \J)\le Cq^{-\min\{m,n\}r +r^2/4}$;
\item for each $i \in [N]$, the family $\f\left(\Pi_{i},\phi_i\right)$ is $(s,q^{-\min\{m,n\}r +r^2/4})$-uncaptureable.
\end{enumerate}
\end{lem}
\begin{proof}
Set $\epsilon: = q^{-\min\{m,n\}r +r^2/4}$. We construct a set $J$ of linear maps such that the statement of the lemma holds with $\J = \langle J \rangle$. We construct $J$ iteratively, along with a labelled, rooted tree $\tree$. Start with $J = \emptyset$, and with $\tree$ consisting of a single node (the root), labelled $v_{\emptyset}$. If $\f$ itself is $(s,\epsilon)$-uncaptureable, then stop, declare $v_{\emptyset}$ to be a good leaf, and take $\J = \Lin(V,W)$. Otherwise, $\f$ is $(s,\epsilon)$-captureable, so there exist linear maps $\Pi:S \to W$ and $\phi: A\to V^*$, with $\dim(S)+\dim(A) \leq s$, such that $\mu\left(\f\left(\overline{\Pi},\overline{\phi}\right)\right)\le \epsilon$. For each $x \in S \setminus \{0\}$, add a new node to $\tree$ which is a child of $v_{\emptyset}$, labelled $v_{\Pi_x}$, where $\Pi_x$ is the linear map from $\Span\{x\}$ to $W$ mapping $x$ to $\Pi(x)$, and for each $a \in A \setminus \{0\}$, add a new node to $\tree$ which is a child of $v_{\emptyset}$, labelled $v_{\phi_a}$, where $\phi_a$ is the linear map from $\Span\{a\}$ to $V^*$ mapping $a$ to $\phi(a)$.

Now at any stage, if $\tree$ has at least one leaf that has not yet been declared good or bad, choose one such. Suppose it is labelled $v_{\Sigma,\psi}$ for some linear maps $\Sigma:T \to W$ and $\psi:B \to V^*$. If $\dim(T)+\dim(B)=r$, then declare $v_{\Sigma,\psi}$ to be a bad leaf of $\tree$. If $\dim(T)+\dim(B)<r$ and $\f(\Sigma,\psi)$ is $(s,\epsilon)$-uncapturable, then add $(\Sigma,\psi)$ to $J$ and declare $v_{\Sigma,\psi}$ to be a good leaf of $\tree$. If $\dim(T)+\dim(B)<r$ and $\f(\Sigma,\psi)$ is $(s,\epsilon)$-capturable, then there exist subspaces $S \leq V$ and $A \leq W^*$ and linear maps $\Pi:S \to W$ and $\phi: A\to V^*$ such that $S \cap T = \{0\}$, $A \cap B = \{0\}$ and $\dim(S) + \dim(A) \leq s$, and $\mu^{|\,\Sigma,\psi}(\f\left(\Sigma,\psi,\overline{\Pi},\overline{\phi})\right)\le\epsilon$ (we call such maps $\Pi$ and $\phi$ {\em capturing maps}). For each $s \in S \setminus \{0\}$, add a new node to $\tree$ which is a child of $v_{\Sigma,\psi}$, labelled $v_{\Sigma_s,\psi}$, where $\Sigma_s$ is the linear map from $\Span(T \cup \{s\})$ to $W$ which agrees with $\Sigma$ on $T$ and maps $s$ to $\Pi(s)$. Similarly, for each $a \in A \setminus \{0\}$, add a new node to $\tree$ which is a child of $v_{\Sigma,\psi}$, labelled $v_{\Sigma,\psi_a}$, where $\psi_a$ is the linear map from $\Span(B \cup \{a\})$ to $V^*$ which agrees with $\psi$ on $B$ and maps $a$ to $\phi(a)$.

This process terminates when all leaves of $\tree$ have been declared good or bad. At this stage, let $\J = \langle J \rangle$. (Note that $J$ consists of all the pairs of linear maps labelling good leaves.) By the definition of `good', $\f(\Pi,\phi)$ is $(s,\epsilon)$-uncaptureable for every $(\Pi,\phi) \in J$. Note that every leaf of $\tree$ has depth at most $r$ (relative to the root), since the depth of a leaf $v_{\sigma}$ is simply the cardinality of the domain of the corresponding bijection $\sigma$. Note also that every node has at most $q^s-1$ children. Hence, the tree $\tree$ has at most $(q^s-1)^r$ leaves, so it has at most $(q^s-1)^r$ good leaves, and therefore $|J| \leq (q^s-1)^r$. Observe that for any linear map $\tau \in \f \setminus \J$, either $\tau$ agrees everywhere with $\Sigma$ and $\tau^*$ agrees everywhere with $\psi$ for some bad leaf $v_{(\Sigma,\psi)}$, or else $\tau$ disagrees (in at least one place) with at least one out of the pair of linear maps labelling some leaf. In the former case, $\tau \in \langle (\Sigma,\psi) \rangle$, and $\mu(\langle \sigma \rangle) = q^{-mr_2-nr_1+r_1r_2} \leq q^{-r\min\{m,n\}+r^2/4}$ (where $r_1: = \dim(\Domain(\Sigma))$ and $r_2: = \dim(\Domain(\psi))$). In the latter case, let $v_{(\Sigma',\psi')}$ be an internal node of maximal depth such that $\Sigma'$ agrees everywhere with $\tau$ and $\psi'$ agrees everywhere with $\tau^*$. Then $\tau \in \f(\Sigma',\psi',\overline{\Pi},\overline{\phi})$, where $\Pi$ and $\phi$ are capturing maps (as defined above), so that $\mu^{|\,\Sigma',\psi'}(\f(\Sigma',\psi',\overline{\Pi},\overline{\phi})) \leq \epsilon$. Since there are at most $(q^s-1)^r$ possibilities for bad leaves and at most $(q^s-1)^d$ possibilities for internal nodes $v_{\Sigma',\psi'}$ of depth $d$ (for each $d < r$), the union bound implies that $\mu(\f \setminus \J) \leq 2q^r(q^s-1)^r \epsilon$, as required.
\end{proof}

\section{Fourier analytic tools}
Our proof relies crucially on considering the Fourier expansion of functions on $\Lin(V,W)$ derived from families of linear maps. We briefly give the relevant background.
\begin{defn}
\label{defn:trace}
Let $q=p^s$ where $p$ is prime. Define
$$\tau:\mathbb{F}_q \to \mathbb{F}_p;\quad \tau(x) = x+x^p+\ldots+x^{p^{s-1}}.$$
\end{defn}
\noindent This is sometimes known as the {\em trace map}, not to be confused with the trace of a matrix. Note that $\tau(x+y) = \tau(x)+\tau(y)$ for all $x,y \in \mathbb{F}_q$, and that $\tau(-x) = -\tau(x)$ for all $x \in \mathbb{F}_q$.
\newline

\noindent For each $X \in \Lin(W,V)$, we define 
$$u_X:\Lin(V,W) \to \mathbb{C};\quad u_{X}(A) = \omega^{\tau(\Trace(XA))},$$
where $\omega : = \exp(2\pi i/p)$. Then the functions $\{u_X:\ X \in \Lin(W,V)\}$ are the characters of the Abelian group $(\Lin(V,W),+)$, so they form an orthonormal basis for $L^2(\Lin(V,W))$. Hence, for 
$f:\Lin(V,W)\to \mathbb{C}$, defining
$$\hat{f}(X) = q^{-\dim(V)\dim(W)} \sum_{\sigma \in \Lin(V,W)} f(\sigma) \overline{u_X(\sigma)},$$
we have the {\em Fourier expansion}
$$f = \sum_{X \in \Lin(W,V)} \hat{f}(X)u_X.$$
The {\em degree} of $f$ is $\max\{\rank(X):\ \hat{f}(X) \neq 0\}$. If $f:\Lin(V,W) \to \mathbb{C}$, we write
$$f^{(= d)} := \sum_{X \in \Lin(W,V):\ \rank(X) =d} \hat{f}(X)u_X$$
and
$$f^{(\leq d)} := \sum_{X \in \Lin(W,V):\ \rank(X) \leq d} \hat{f}(X)u_X.$$

\noindent Given $V' \leq V,W' \leq W$, $f:\Lin(V,W)\to \mathbb{C}$, we define projection maps
$$\Pi_{V'}\left(f\right):=\sum_{{X\in \Lin(W,V):}\atop{V'= \mathrm{im\left(X\right)}}}\hat{f}\left(X\right)u_{X},$$
and
$$\Pi_{W'}\left(f\right):=\sum_{{X\in \Lin(W,V):}\atop{\mathrm{ker}\left(X\right)= W'}}\hat{f}\left(X\right)u_{X}.$$
We write ${V \brack d}$ for the set of $d$-dimensional subspaces of $V$, and, by a slight abuse of notation, ${W\brack -d}$ for the set of codimension-$d$ subspaces of $W$. As usual, for integers $m,d$ and a prime power $q$, the Gaussian binomial coefficient is defined by
$${m \brack d}_q : = \frac{\prod_{i=1}^{d}(q^{m-i+1} - 1)}{\prod_{i=1}^{d}(q^{d-i+1}-1)},$$
which is the number of $d$-dimensional subspaces of $\mathbb{F}_q^m$. We will sometimes make use of the following lower bound on the Gaussian binomial coefficients, which follows immediately from the above formula (or which can also be seen by considering the column space of an $m$ by $d$ matrix with entries in $\mathbb{F}_q$, whose top-left $d$ by $d$ submatrix consists of the identity matrix):
\begin{equation}\label{eq:trivial-lower} {m \brack d}_q \geq q^{d(m-d)}.\end{equation}

We will need the following hypercontractive inequality.

\begin{prop}
\label{prop:higher-norms}
Let $d \in \mathbb{N}$, let $m = \dim(V)$ and $n =dim(W)$, let $f:\Lin(V,W)\to \mathbb{C}$ be a function of degree $d$, and let $k \geq 4$ be even. Then 
\begin{equation}\label{eq:target}
\mathbb{E}\left[|f^{(=d)}|^k\right]\leq k^7d^6 q^{k^3d^2/2} q^{(3k/4-1)d\max\{m,n\}} \left(\sum_{V'\leq V:\atop \dim(V') = d}\left\Vert \Pi_{V'}\left(f\right)\right\Vert _{2}^{k} + \sum_{W'\leq W:\atop \codim(W') = d} \left\Vert \Pi_{W'}\left(f\right)\right\Vert _{2}^{k}\right).\end{equation}
\end{prop}

This can be viewed as a (restricted) hypercontractive inequality, because it bounds the higher norms of low-degree functions in terms of the 2-norms of certain of its projections (the latter will later be bounded using quasiregularity). It is a `restricted' rather than a `full' hypercontractive inequality because the 2-norms of these projections are not necessarily small for arbitrary low-degree functions (the quasireguality hypothesis, or something similar, is needed). We remark that the bound in Proposition \ref{prop:higher-norms} is suboptimal, but it suffices for our purposes. A more precise (and indeed, essentially sharp) result is proved in \cite{e-kindler-l}; while we could use Theorem 4 therein as a black box in place of Proposition \ref{prop:higher-norms}, the proof of the former is much harder (and longer) than the latter, so we have instead opted to make this paper self-contained by including the proof of Proposition \ref{prop:higher-norms}.

The proof of Proposition \ref{prop:higher-norms} relies upon the following (straightforward) linear-algebraic claim. For $d \in \mathbb{N}$ we write
$$\mathcal{L}_{d}(W,V): = \{X \in \Lin(W,V):\ \rank(X)=d\}$$
and
$$\mathcal{L}_{\leq d}(W,V): = \{X \in \Lin(W,V):\ \rank(X)\leq d\}.$$
\begin{claim}
\label{claim:sum-rank-nullity} 
Fix $\lambda_1,\ldots,\lambda_r \in \mathbb{F}_q \setminus \{0\}$, and let
$$\mathcal{A}: = \{(X_i)_{i=1}^{r} \in (\mathcal{L}_1(W,V))^r:\ \sum_{i=1}^{r}\lambda_i X_i = 0\}.$$
For any $(X_i)_{i=1}^{r} \in \A$, we have
$$\dim\left(\sum_i \Image(X_i)\right) + \codim\left(\bigcap_i \ker(X_i)\right) \leq r.$$
\end{claim}
\begin{proof}[Proof of Claim \ref{claim:sum-rank-nullity}.]
Let $s = \dim\left(\sum_i \Image(X_i)\right)$; then by reordering the $X_i$ if necessary, we may assume that $\sum_{i=1}^{r} \Image(X_i) = \bigoplus_{i=1}^{s} \Image(X_i) \in {V \brack s}$. By assumption, $\Image(X_1),\ldots,\Image(X_s)$ are linearly independent. If $v \in \ker(\sum_{i=1}^{s} \lambda_i X_i)$ then for any $j \in [s]$ we have
$$\Image(X_j) \ni \lambda_j X_j(v) = -\sum_{1 \leq i \leq s,\atop i \neq j} \lambda_i X_i(v) \in \bigoplus_{1 \leq i \leq s,\atop i \neq j} \Image(X_j),$$
so $v \in \ker(X_j)$; it follows that 
$$\bigcap_{i=1}^{s} \ker(X_i) = \ker\left(\sum_{i=1}^{s}  \lambda_i X_i\right) = \ker\left(\sum_{i=s+1}^{r} \lambda_i X_i\right) \supset \bigcap_{i=s+1}^{r} \ker(X_i),$$
so
$$\bigcap_{i=1}^{r} \ker(X_i) = \bigcap_{i=s+1}^{r} \ker(X_i),$$
and therefore $\codim(\bigcap_{i=1}^{r} \ker(X_i)) = \codim(\bigcap_{i=s+1}^{r} \ker(X_i)) \leq r-s$, as required.
\end{proof}

\begin{proof}[Proof of Proposition \ref{prop:higher-norms}.]
Let $d \in \mathbb{N}$, let $f:\Lin(V,W)\to \mathbb{C}$ be a function of degree $d$, and let $k \geq 4$ be even. We may clearly assume that $f = f^{(=d)}$. Observe that
$$\mathbb{E}[|f|^k] = \mathbb{E}[f^{k/2} \overline{f}^{k/2}] = \sum_{(X_i)_{i=1}^{k}\in \mathcal{A}} \prod_{i=1}^{k} c^{i-1}(\hat{f}(X_i)),$$
where
$$\mathcal{A} :=\left\{(X_i)_{i=1}^{k}:\ X_i \in \Lin_d(W,V)\ \forall i,\ \sum_{i=1}^{k}(-1)^{i-1}X_i = 0\right\},$$
and $c(g) = \overline{g}$ denotes complex conjugation, so that $c^{j}(g) = g$ for even $j$ and $c^j(g) = \overline{g}$ for odd $j$. Since the right-hand side of (\ref{eq:target}) is unchanged under replacing $\hat{f}(X)$ with $|\hat{f}(X)|$ for each $X$, and the left-hand side of (\ref{eq:target}) can only increase under this operation, we may assume that $\hat{f}(X) \in \mathbb{R}_{\geq 0}$ for all $X$. In this case, we have
$$\mathbb{E}[|f|^k]  = \sum_{(X_i)_{i=1}^{k}\in \mathcal{A}} \prod_{i=1}^{k} \hat{f}(X_i),$$
so it suffices to prove that
$$\sum_{(X_i)_{i=1}^{k}\in \mathcal{A}} \prod_{i=1}^{k} \hat{f}(X_i) \leq k^7d^6 q^{k^3d^2/2}q^{(3k/4-1)d\max\{m,n\}} \left(\sum_{V'\leq V:\atop \dim(V') = d}\left\Vert \Pi_{V'}\left(f\right)\right\Vert _{2}^{k} + \sum_{W'\leq W:\atop \codim(W') = d} \left\Vert \Pi_{W'}\left(f\right)\right\Vert _{2}^{k}\right).$$
By the Cauchy-Schwarz inequality, we have
$$\sum_{(X_i)_{i=1}^{k}\in \A} \prod_{i=1}^{k} \hat{f}(X_i) \leq \sqrt{\sum_{(X_i)_{i=1}^{k}\in \A} \prod_{i=1}^{k/2} \hat{f}(X_i)^2}\sqrt{\sum_{(X_i)_{i=1}^{k}\in \A} \prod_{i=k/2+1}^{k} \hat{f}(X_i)^2}.$$
For each $\mathbf{X} = (X_i)_{i=1}^{k}\in \mathcal{A}$, let
\begin{align*} D_1(\mathbf{X}) & = \dim\left(\sum_{i=1}^{k/2} \Image(X_i)\right),\\
D_2(\mathbf{X}) & = \dim\left(\sum_{i=k/2+1}^{k} \Image(X_i)\right),\\
D_3(\mathbf{X}) & = \dim\left(\left(\sum_{i=1}^{k/2} \Image(X_i)\right) \cap \left(\sum_{i=k/2+1}^{k} \Image(X_i)\right)\right),\\
E_1(\mathbf{X}) & = \codim \left(\bigcap_{i=1}^{k/2} \ker(X_i)\right),\\
E_2(\mathbf{X}) & = \codim \left(\bigcap_{i=k/2+1}^{k} \ker(X_i)\right),\\
E_3(\mathbf{X}) & = \codim \left(\left(\bigcap_{i=1}^{k/2} \ker(X_i) \right)+\left(\bigcap_{i=k/2+1}^{k}\ker(X_i)\right)\right).
\end{align*}
Note that if $X \in \Lin_d(W,V)$ then there exist $Y_1,\ldots,Y_d \in \Lin_{1}(W,V)$ such that
$$X = \sum_{i=1}^{d}Y_i,$$
which implies that $\Image(X) = \sum_{i=1}^{d} \Image(Y_i)$ and $\ker(X) = \cap_{i=1}^{d}\ker(Y_i)$. Hence, we may apply Claim \ref{claim:sum-rank-nullity} (with $r=dk$) to conclude that for any $\mathbf{X} = (X_i)_{i=1}^{k} \in \A$, we have
$$\dim\left(\sum_{i=1}^{k} \Image(X_i)\right)+\codim\left(\bigcap_{i=1}^{k} \ker(X_i)\right) \leq dk.$$
It follows that
$$D_1(\mathbf{X})+D_2(\mathbf{X})-D_{3}(\mathbf{X}) + E_1(\mathbf{X})+E_2(\mathbf{X})-E_{3}(\mathbf{X}) \leq dk\quad \forall \mathbf{X} \in \A.$$
For each $(\mathbf{d},\mathbf{e}) = (d_1,d_2,d_3,e_1,e_2,e_3) \in (\mathbb{N} \cup \{0\})^6$ such that $d_1+d_2-d_3+e_1+e_2-e_3 \leq kd$, define
$$\A(\mathbf{d},\mathbf{e}) = \A(d_1,d_2,d_3,e_1,e_2,e_3): = \{\mathbf{X} \in \A:\ D_i(\mathbf{X}) = d_i,\ E_i(\mathbf{X}) = e_i\ \forall i \in [3]\}.$$
For each $(X_i)_{i=1}^{k/2} \in (\Lin_d(W,V))^{k/2}$ such that $\dim(\sum_{i=1}^{k/2}\Image(X_i)) = d_1$ and $\codim(\cap_{i=1}^{k/2}\ker(X_i)) =e_1$, there are at most 
$$q^{d_1d_3+e_1e_3+d_2e_2k/2} |V|^{d_2-d_3}|W|^{e_2-e_3}$$
choices of $(X_i)_{i=k/2+1}^{k} \in (\Lin_d(W,V))^{k/2}$ such that $\mathbf{X} = (X_i)_{i=1}^{k} \in \A$ and $\mathbf{X}$ satisfies $D_i(\mathbf{X})=d_i$ $(i=1,2,3)$ and $E_i(\mathbf{X}) = e_i$ $(i=1,2,3)$. Indeed, there are at most
$${d_1 \brack d_3}_q {m-d_3 \brack d_2-d_3}_q \leq q^{d_1 d_3} |V|^{d_2-d_3}$$
$d_2$-dimensional subspaces $U$ of $V$ with
$$\dim\left( \left(\sum_{i=1}^{k/2}\Image(X_i)\right) \cap U\right) = d_3,$$
and at most 
$${e_1 \brack e_3}_q {n-e_3 \brack e_2-e_3}_q \leq q^{e_1e_3} |W|^{e_2-e_3}$$
$e_2$-codimensional subspaces $U'$ of $W$ with
$$\codim\left(\left(\bigcap_{i=1}^{k/2}\ker(X_i)\right)+U'\right) = e_3,$$
and given such $U,U'$, there are at most $q^{d_2e_2k/2}$ choices of $(X_i)_{i=k/2+1}^{k} \in (\Lin_d(W,V))^{k/2}$ with
$$\sum_{i=k/2+1}^{k} \Image(X_i) = U,\quad \bigcap_{i=k/2+1}^{k} \ker(X_i) = U'.$$
By exactly the same argument, for each $(X_i)_{i=k/2+1}^{k} \in (\Lin_d(W,V))^{k/2}$ such that $\dim(\sum_{i=1}^{k/2}\Image(X_i)) = d_2$ and $\codim(\cap_{i=k/2}^{k}\ker(X_i)) =e_2$, there are at most $$q^{d_2d_3+e_2e_3+d_1e_1k/2} |V|^{d_1-d_3}|W|^{e_1-e_3}$$
choices of $(X_i)_{i=1}^{k/2} \in (\Lin_d(W,V))^{k/2}$ such that $\mathbf{X} = (X_i)_{i=1}^{k} \in \A$ and $\mathbf{X}$ satisfies $D_i(\mathbf{X})=d_i$ $(i=1,2,3)$ and $E_i(\mathbf{X}) = e_i$ $(i=1,2,3)$. Writing
$$\mathcal{B}(r,s): = \left\{ (X_i)_{i=1}^{k/2} \in (\Lin_d(W,V))^{k/2}:\ \dim\left(\sum_{i=1}^{k/2}\Image(X_i)\right) = r,\ \codim\left(\bigcap_{i=1}^{k/2} \ker(X_i)\right) = s\right\}$$
for each $r,s \in \mathbb{N} \cup \{0\}$, it follows that
\begin{align*} \sum_{\mathbf{X}\in \A(\mathbf{d},\mathbf{e})} \prod_{i=1}^{k} \hat{f}(X_i) & \leq q^{d_1d_3/2+d_2d_3/2+e_1e_3/2+e_2e_3/2+d_1e_1k/4+d_2e_2k/4} \cdot |V|^{(d_1+d_2)/2-d_3}|W|^{(e_1+e_2)/2-e_3} \\
& \cdot \sqrt{\sum_{(X_i)_{i=1}^{k/2} \in \mathcal{B}(d_1,e_1)} \prod_{i=1}^{k/2} \hat{f}(X_i)^2} \sqrt{\sum_{(X_i)_{i=1}^{k/2} \in \mathcal{B}(d_2,e_2)} \prod_{i=1}^{k/2} \hat{f}(X_i)^2}.\end{align*}
Now observe that
$$\sum_{(X_i)_{i=1}^{k/2} \in \mathcal{B}(d_1,e_1)} \prod_{i=1}^{k/2} \hat{f}(X_i)^2 \leq \min\left\{\sum_{(U_i)_{i=1}^{k/2} \in \left({V \brack d}\right)^{k/2}:\atop \dim(\sum_{i=1}^{k/2}U_i)=d_1} \prod_{i=1}^{k/2} \mathbb{E}[|f_{U_i}|^2],\sum_{(U_i')_{i=1}^{k/2} \in \left({W \brack -d}\right)^{k/2}:\atop \codim(\cap_{i=1}^{k/2}U_i')=e_1} \prod_{i=1}^{k/2} \mathbb{E}[|f_{U_i'}|^2]\right\},$$
where for each $U \in {V \brack d}$ we define $f_U: = \Pi_U(f) = \sum_{X \in \Lin(W,V):\ \Image(X)=U}\hat{f}(X)u_X$, and for each $U' \in {W \brack -d}$ we define $f_{U'}: = \Pi_{U'}(f) = \sum_{X \in \Lin(W,V):\ \ker(X)=U'}\hat{f}(X)u_X$. For $d_1 \leq e_1$ we will bound the left-hand side of the above expression using the first term of the minimum; for $d_1 > e_1$ we will use the second term. To this end, we first bound from above the first term of the minimum. 

Given $(U_i)_{i=1}^{k/2} \in ({V \brack d})^{k/2}$ such that $\dim(\sum_{i=1}^{k/2}U_i)=d_1$, let $j_1 \in [k/2]$ be such that
$$\mathbb{E}[|f_{U_{j_1}}|^2] = \max_{i \in [k/2]}\mathbb{E}[|f_{U_i}|^2];$$
we then have
$$\prod_{i=1}^{k/2} \mathbb{E}[|f_{U_i}|^2] \leq (\mathbb{E}[|f_{U_{j_1}}|^2])^{k/2}.$$
Given $U \in {V \brack d}$, there are at most
$$(k/2) {m -d \brack d_1-d}_q \left({d_1 \brack d}_q\right)^{k/2-1} \leq (k/2)q^{dd_1(k/2-1)} |V|^{d_1-d}$$
choices of $(U_i)_{i=1}^{k/2} \in ({V \brack d})^{k/2}$ such that $U_{j_1}=U$ for some $j_1 \in [k/2]$ and $\dim(\sum_{i=1}^{k/2}U_i) = d_1$; it follows that
$$\sum_{(U_i)_{i=1}^{k/2} \in \left({V \brack d}\right)^{k/2}:\atop \dim(\sum_{i=1}^{k/2}U_i)=d_1} \prod_{i=1}^{k/2} \mathbb{E}[|f_{U_i}|^2] \leq (k/2)q^{dd_1(k/2-1)} |V|^{d_1-d} \sum_{U \in {V\brack d}} (\mathbb{E}[|f_U|^2])^{k/2}.$$
By a very similar argument, the second term in the minimum can be bounded similarly:
$$\sum_{(U_i')_{i=1}^{k/2} \in \left({W \brack -d}\right)^{k/2}:\atop \codim(\cap_{i=1}^{k/2}U_i')=e_1} \prod_{i=1}^{k/2} \mathbb{E}[|f_{U_i'}|^2] \leq (k/2)q^{de_1(k/2-1)} |W|^{e_1-d} \sum_{U' \in {W\brack -d}} (\mathbb{E}[|f_{U'}|^2])^{k/2}.$$
Putting everything together, we obtain
\begin{align*} \sum_{\mathbf{X}\in \A(\mathbf{d},\mathbf{e})} \prod_{i=1}^{k} \hat{f}(X_i) \leq q^{d_1d_3/2+d_2d_3/2+e_1e_3/2+e_2e_3/2+d_1e_1k/4+d_2e_2k/4} \cdot |V|^{(d_1+d_2)/2-d_3}|W|^{(e_1+e_2)/2-e_3} \\
\cdot \sqrt{\min\left\{(k/2)q^{dd_1(k/2-1)}|V|^{d_1-d} \sum_{U \in {V\brack d}} (\mathbb{E}[|f_U|^2])^{k/2},\ (k/2)q^{de_1(k/2-1)}|W|^{e_1-d} \sum_{U' \in {W\brack -d}} (\mathbb{E}[|f_{U'}|^2])^{k/2}\right\}}\\
\cdot \sqrt{\min\left\{(k/2)q^{dd_2(k/2-1)}|V|^{d_2-d} \sum_{U \in {V\brack d}} (\mathbb{E}[|f_U|^2])^{k/2},\ (k/2)q^{de_2(k/2-1)}|W|^{e_2-d} \sum_{U' \in {W\brack -d}} (\mathbb{E}[|f_{U'}|^2])^{k/2}\right\}}.\end{align*}
In the case where $d_1 \leq e_1$ and $d_2 \leq e_2$, we use the first terms of both minima, obtaining the bound
\begin{align*} \sum_{\mathbf{X}\in \A(\mathbf{d},\mathbf{e})} \prod_{i=1}^{k} \hat{f}(X_i) & \leq C_{d,k,q} \cdot |V|^{d_1+d_2-d_3-d}|W|^{(e_1+e_2)/2-e_3} \sum_{U \in {V\brack d}} (\mathbb{E}[|f_U|^2])^{k/2}\\
& \leq C_{d,k,q} \cdot M^{d_1+d_2-d_3+e_1+e_2-e_3-\max\{d_1,e_1\}/2 - \max\{d_2,e_2\}/2-d} \sum_{U \in {V\brack d}} (\mathbb{E}[|f_U|^2])^{k/2},
\end{align*}
where we define $M: = \max\{|V|,|W|\}$ and
$$C_{d,k,q}: = (k/2)q^{d_1d_3/2+d_2d_3/2+e_1e_3/2+e_2e_3/2+d_1e_1k/4+d_2e_2k/4+d(d_1+d_2+e_1+e_2)k/4} \leq (k/2)q^{k^3d^2/2}.$$
In the case where $d_1 > e_1$ and $d_2 > e_2$ we use the second terms of both minima, obtaining the bound
\begin{align*} \sum_{\mathbf{X}\in \A(\mathbf{d},\mathbf{e})} \prod_{i=1}^{k} \hat{f}(X_i) & \leq C_{d,k,q} \cdot |V|^{(d_1+d_2)/2-d_3}|W|^{e_1+e_2-e_3-d} \sum_{U' \in {W\brack -d}} (\mathbb{E}[|f_{U'}|^2])^{k/2}\\
& \leq C_{d,k,q} \cdot M^{d_1+d_2-d_3+e_1+e_2-e_3-\max\{d_1,e_1\}/2 - \max\{d_2,e_2\}/2-d} \sum_{U' \in {W\brack -d}} (\mathbb{E}[|f_{U'}|^2])^{k/2}.
\end{align*}
In the case where $d_1 \leq e_1$ and $d_2 > e_2$, we use the first term of the first minimum and the second term of the second, obtaining the bound
\begin{align*} \sum_{\mathbf{X}\in \A(\mathbf{d},\mathbf{e})} \prod_{i=1}^{k} \hat{f}(X_i) & \leq C_{d,k,q} \cdot |V|^{d_1+d_2/2-d_3-d/2}|W|^{e_1/2+e_2-e_3-d/2} \sqrt{\sum_{U \in {V\brack d}} (\mathbb{E}[|f_U|^2])^{k/2}} \sqrt{\sum_{U' \in {W\brack -d}} (\mathbb{E}[|f_{U'}|^2])^{k/2}}\\
& \leq C_{d,k,q} \cdot |V|^{d_1+d_2/2-d_3-d/2}|W|^{e_1/2+e_2-e_3-d/2} \left(\sum_{U \in {V\brack d}} (\mathbb{E}[|f_U|^2])^{k/2}+\sum_{U' \in {W\brack -d}} (\mathbb{E}[|f_{U'}|^2])^{k/2}\right)\\
& \leq C_{d,k,q} \cdot M^{d_1+d_2-d_3+e_1+e_2-e_3-\max\{d_1,e_1\}/2 - \max\{d_2,e_2\}/2-d}\\
&\cdot \left(\sum_{U \in {V\brack d}} (\mathbb{E}[|f_U|^2])^{k/2}+\sum_{U' \in {W\brack -d}} (\mathbb{E}[|f_{U'}|^2])^{k/2}\right)
\end{align*}
using the AM/GM inequality. In the last case ($d_1 > e_1$ and $d_2 \leq e_2$), we obtain the same bound as in the previous case, by symmetry. Writing
$$S: = d_1+d_2-d_3+e_1+e_2-e_3,$$
we have $S \leq kd$ and $\max\{d_1,e_1\}+\max\{d_2,e_2\} \geq \tfrac{1}{2}(d_1+d_2+e_1+e_2) = \tfrac{1}{2}(S+d_3+e_3)$, so 
\begin{align*} d_1+d_2-d_3+e_1+e_2-e_3-\max\{d_1,e_1\}/2 - \max\{d_2,e_2\}/2-d & \leq S - \tfrac{1}{4}S -d_3/4-e_3/4-d\\
& \leq 3S/4-d\\
& \leq 3kd/4-d.
\end{align*}
Summing over all $\leq (kd/2)^6$ possible choices of $(\mathbf{d},\mathbf{e})$ completes the proof.
\end{proof}

We now need the following easy lemma, relating the $2$-norms of the projections $\Pi_{V'}$ and $\Pi_{W'}$ (which appear in Proposition \ref{prop:higher-norms}) to quasiregularity.

\begin{lem}
\label{lem:qr-proj}
Let $f:\Lin(V,W) \to \mathbb{R}_{\geq 0}$ be $(s,C)$-quasiregular, where $C \geq 1$. Then for any $V' \leq V$ with $\dim(V') \leq s$ we have
$$\|\Pi_{V'}(f)\|_2^2 \leq C^2 (\mathbb{E}[f])^2,$$
and for any $W' \leq W$ with $\codim(W') \leq s$ we have
$$\|\Pi_{W'}(f)\|_2^2 \leq C^2 (\mathbb{E}[f])^2.$$
\end{lem}

\begin{proof}
We prove the second statement only; the first follows by a duality argument. This statement is trivial if $W'=W$, so we may assume that $\codim(W') \geq 1$. Without loss of generality, we may assume that
$$W' = \Span\{e_{d+1},\ldots,e_n\}$$
for some $d \in \mathbb{N}$ with $d \leq s$. Then $\Pi_{W'}(f)$ is precisely the orthogonal projection of $f$ onto
$$\{u_{X}:X \in \mathcal{M}_d\},$$
where $\mathcal{M}_d$ denotes the set of all matrices supported on the first $d$ columns, and having rank $d$. Note that $\mathcal{M}_d \subset \mathcal{C}_d$, where $\mathcal{C}_d$ denotes the set of all nonzero matrices supported on the first $d$ columns. Observe that for all $X \in \mathcal{C}_d$, we have
$$\hat{f}(X) = \hat{g}(X),$$
where $g:\Lin(V,W) \to \mathbb{R}_{\geq 0}$ is defined by setting $g(A)$ to be the average value of $f(B)$ over all matrices $B$ such that the matrix formed by the first $d$ columns of $B$ is equal to the matrix formed by the first $d$ columns of $A$. Hence, we have
$$\|\Pi_{W'}(f)\|_2^2 = \sum_{X \in \mathcal{M}_d} \hat{f}(X)^2 \leq \sum_{X \in \mathcal{C}_d}\hat{f}(X)^2 = \sum_{X \in \mathcal{C}_r} \hat{g}(X)^2.$$
Now note that the right-hand side is simply the variance of the function $g$ (since $g(A)$ depends only upon the first $d$ columns of $A$), which, by the quasiregularity hypothesis, is at most $(C^2-1) (\mathbb{E}[f])^2$ (since $g(A) \in [0,C \mathbb{E}[f]]$ for all $A$). The second part of the lemma follows.

\end{proof}

We now describe how we will use the preceding two results. By Proposition \ref{prop:higher-norms} and Lemma \ref{lem:qr-proj}, if $k \geq 4$ is even, $d \leq s$, $C \geq 1$ and $f:\Lin(V,W) \to \mathbb{R}$ is $(s,C)$-quasiregular, then we have 
\[
\mathbb{E}\left[|f^{(= d)}|^k\right]\leq k^7d^7 q^{k^3d^2/2} q^{3dk\max \{m,n\}/4} C^k (\mathbb{E}[f])^k,
\]
so taking $k$th roots, we have
\begin{equation}\label{eq:higher-norm-bound}
\|f^{(= d)}\|_k \leq O(1) q^{k^2d^2} q^{3d\max\{m,n\}/4}C \mathbb{E}[f].
\end{equation}
By H\"older's inequality, if $f$ is $[0,1]$-valued we have
$$\mathbb{E}[|f^{(= d)}|^2] = \mathbb{E}[f \cdot f^{(= d)}] \leq \|f\|_{k/(k-1)} \cdot \|f^{(= d)}\|_k \leq (\mathbb{E}[f])^{1-1/k} \|f^{(= d)}\|_k.$$
Substituting (\ref{eq:higher-norm-bound}) into the above gives
\begin{equation} \label{eq:level-d} \mathbb{E}[|f^{(= d)}|^2]  \leq O(1) q^{k^2d^2} q^{3\max\{m,n\}d/4}C(\mathbb{E}[f])^{2-1/k}.\end{equation}
This says that if $f$ is an $(s,C)$-quasiregular function of small expectation, then the degree-at-most-$s$-part of $f$ has small $L^2$-norm. It makes our spectral argument much simpler, since the `bad' eigenspaces will correspond to low-degree parts of functions.

\section{Spectral tools}

For each $t \in \mathbb{N} \cup \{0\}$, let $M_t$ be the normalized adjacency matrix of the Cayley graph $\Gamma_t$ on $(\Lin(V,W),+)$ generated by the set $\mathcal{I}_t(V,W)$ of all linear maps from $V$ to $W$ whose kernel has dimension $t$, or equivalently whose rank is $m-t$ (where $M_t$ is normalized so that all its row-sums are 1). Clearly, a $t$-intersection-free family $\mathcal{F} \subset \Lin(V,W)$ is precisely an independent set in the graph $\Gamma_t$; to analyse such independent sets we will require bounds on the eigenvalues of $\Gamma_t$.

It is easy to see that the eigenspaces of $M_t$ are
$$U_d : = \Span\{u_{X}: X \in \Lin(W,V):\ \rank(X)=d\} \quad (0 \leq d \leq \min\{m,n\}).$$
Indeed, since $\Gamma_t$ is a Cayley graph on an Abelian group, its eigenvectors are precisely the characters $(u_X)_{X \in \mathcal{L}(W,V)}$, and the eigenvalue $\lambda_X$ corresponding to $u_X$ is given by
$$\lambda_X = \sum_{A \in \Lin(V,W):\atop \rank(A) = m-t}u_X(A) = \sum_{A \in \Lin(V,W):\atop \rank(A)=m-t}\omega^{\tau(\Trace(XA))};$$
if $X,X' \in \Lin(W,V)$ have the same rank, then there exist invertible linear maps $P \in \Lin(V,V)$ and $Q \in \Lin(W,W)$ such that $X' = PXQ$, so
\begin{align*}
    \lambda_{X'} &= \sum_{A \in \Lin(V,W):\atop \rank(A)=m-t}\omega^{\tau(\Trace(X'A))}\\
    & = \sum_{A \in \Lin(V,W):\atop \rank(A)=m-t}\omega^{\tau(\Trace(PXQA))}\\
    & = \sum_{A \in \Lin(V,W):\atop \rank(A)=m-t}\omega^{\tau(\Trace(XQAP))}\\
    & = \sum_{A \in \Lin(V,W):\atop \rank(A)=m-t}\omega^{\tau(\Trace(XA))}\\
    & = \lambda_X.
    \end{align*}
    Let $\lambda^{(t)}_d$ be the eigenvalue of $M_t$ corresponding to $U_d$. First, for clarity, we focus on the case $t=0$. Let us write $\lambda_d^{(0)}: = \lambda_d$ for each $0 \leq d \leq \min\{m,n\}$. We have the following.
\begin{lem}
\label{lem:eval-bound-1}
Let $m = \dim(V) \leq \dim(W)=n$. For each $0 \leq d \leq \min\{m,n\}$, we have
$$|\lambda_d| = O(q^{-(m+n-d)d/2}).$$
\end{lem}
\begin{proof}
Let us define
$$\phi(m,n) : = \frac{|\mathcal{I}_0(V,W)|}{|\Lin(V,W)|} = \frac{\prod_{i=1}^{m}(q^n-q^{i-1})}{q^{nm}} = \prod_{i=1}^{m}(1-q^{-(n-i+1)});$$
note that
$$\phi(m,n) \geq \phi(n,n) = \prod_{i=1}^{n}(1-q^{-i}) \geq \prod_{i=1}^{n} (1-2^{-i}) > \tfrac{1}{4} \quad \forall m \leq n,$$
using the fact that
$$\log_2\left(\prod_{i=1}^{n} (1-2^{-i})\right) = \sum_{i=1}^{n} \log_2(1-2^{-i}) \geq -\sum_{i=1}^{n} 2 \cdot 2^{-i} > -2.$$
Using the well-known trace formula, we have
\begin{equation}\label{eq:sum-squares-upper} \sum_{d=0}^{n} \dim(U_d) \lambda_d^2 = \Trace(M_0^2) = \frac{2e(\Gamma)}{|\mathcal{I}_0(V,W)|^2}= \frac{|\mathcal{I}_0(V,W)|q^{mn}}{|\mathcal{I}_0(V,W)|^2} = \frac{1}{\phi(m,n)} < 4.\end{equation}
Note that $\dim(U_d)$ is precisely the number of rank $d$ matrices in $\Lin(W,V)$. For a rank $d$ matrix $X \in \Lin(W,V)$ there are ${m \brack d}_q$ choices for $\Image(X)$ and ${n \brack d}_q$ choices for $\ker(X)$, and given $\Image(X)$ and $\ker(X)$ there are 
$$|\GL_d(\mathbb{F}_q)| = \prod_{i=1}^{d}(q^d-q^{i-1})$$
choices for $X$, so we have
\begin{equation}\label{eq:dim-lower} \dim(U_d) = {m \brack d}_q {n \brack d}_q \prod_{i=1}^{d} (q^d-q^{i-1})= \prod_{i=1}^{d}\frac{(q^m-q^{i-1})(q^n-q^{i-1})}{q^d - q^{i-1}} = \Omega(q^{(m+n-d)d}),\end{equation}
and therefore, combining (\ref{eq:sum-squares-upper}) and (\ref{eq:dim-lower}), we have
$$|\lambda_d| = O(q^{-(m+n-d)d/2}),$$
as required.
\end{proof} 
    
For general $t$, we have the following.
\begin{lem}
\label{lem:eval-bound-2}
Let $m = \dim(V)$ and $n=\dim(W)$, where $n \geq m-t$. For each $0 \leq d \leq \min\{m,n\}$, we have
$$|\lambda^{(t)}_d| = O(q^{-((m+n-d)d - (n-m)t - t^2)/2}).$$
\end{lem}
\begin{proof}
Let us define
$$\phi(m,n,t) : = \frac{|\mathcal{I}_t(V,W)|}{|\Lin(V,W)|} = {m \brack t} \frac{\prod_{i=1}^{m-t}(q^n-q^{i-1})}{q^{nm}};$$
Using (\ref{eq:trivial-lower}) and the bounds for $\phi(\cdot,\cdot)$ in the proof of the preceding lemma, we have
$$\phi(m,n,t) \geq q^{t(m-t)}\cdot \frac{\prod_{i=1}^{m-t}(q^n-q^{i-1})}{q^{n(m-t)}} \cdot q^{-nt}\geq q^{-(n-m)t-t^2} \phi(m-t,n)> \tfrac{1}{4}q^{-t^2} q^{-(n-m)t}.$$
Again using the trace formula, we have
\begin{equation}\label{eq:sum-squares-upper-2} \sum_{d=0}^{n} \dim(U_d) (\lambda^{(t)}_d)^2 = \Trace(M_t^2) = \frac{2e(\Gamma_t)}{|\mathcal{I}_t(V,W)|^2}= \frac{|\mathcal{I}_t(V,W)|q^{mn}}{|\mathcal{I}_t(V,W)|^2} = \frac{1}{\phi(m,n,t)} = O(1)q^{t^2+(n-m)t}.\end{equation}
Combining (\ref{eq:dim-lower}) with (\ref{eq:sum-squares-upper-2}) yields
$$|\lambda^{(t)}_d| = O(q^{-((m+n-d)d - (n-m)t - t^2)/2}),$$
as required.
\end{proof}

The following is a special case of the result we need; we include its proof to illustrate the idea in a simpler context.
\begin{lem} 
Let $s \in \mathbb{N}$. Suppose that $\f,\g \subset \Lin(V,W)$ are cross-intersecting and $(s,q^{m/16})$-quasiregular, where $m = \dim(V) \leq \dim(W)=n$ and $m \geq 3n/4$. Then for $n$ sufficiently large depending on $s$ we have
$$\min\{\mu(\f),\mu(\g)\} = O(1) q^{-(m+n-s-1)(s+1)/2}.$$
\end{lem}
\begin{proof}
Suppose that $\f,\g \subset \Lin(V,W)$ are cross-intersecting and $(s,C)$-quasiregular, where $C: = q^{m/16}$, and suppose for a contradiction that $\min\{\mu(\f),\mu(\g)\} \geq C_0 q^{-(m+n-s-1)(s+1)/2}$ (for some absolute constant $C_0\geq 1$ to be chosen later), and that $n \geq m \geq 3n/4$. By Lemma \ref{lem:eval-bound-1}, for any $s \in \mathbb{N}$ we have
$$|\lambda_d| \leq O(q^{-(m+n-s-1)(s+1)/2})\quad \forall s+1 \leq d \leq \min\{m,n\},$$
using the fact that $d \mapsto (m+n-d)d$ is an increasing function, for $0 \leq d \leq (m+n)/2$. Writing $f =1_{\f}$ and $g = 1_{\g}$ and applying (\ref{eq:level-d}) and the Cauchy-Schwarz inequality, we have
$$|\langle f^{(=d)},g^{(=d)}\rangle| \leq \|f^{(=d)}\|_2 \|g^{(=d)}\|_2 \leq O(1)q^{k^2d^2} q^{3nd/4}C(\mathbb{E}[f] \mathbb{E}[g])^{1-1/(2k)}.$$
Hence, we have
\begin{align*} 0& =\langle f,Mg \rangle\\
 &= \mathbb{E}[f] \mathbb{E}[g] + \sum_{d=1}^{n} \lambda_d \langle f^{(=d)},g^{(=d)}\rangle \\
& = \mathbb{E}[f]\mathbb{E}[g] + \sum_{d=1}^{s} \lambda_d \langle f^{(=d)},g^{(=d)}\rangle - O(1)q^{-(m+n-s-1)(s+1)/2}\sqrt{\mathbb{E}[f]\mathbb{E}[g]}\\
& \geq \mathbb{E}[f]\mathbb{E}[g] - O(1)\sum_{d=1}^{s} q^{k^2d^2} q^{-(m+n-d)d/2} q^{3nd/4}C(\mathbb{E}[f] \mathbb{E}[g])^{1-1/(2k)}\\
& - O(1)q^{-(m+n-s-1)(s+1)/2}\sqrt{\mathbb{E}[f]\mathbb{E}[g]}\\
& >0,
\end{align*}
provided $\min\{\mathbb{E}[f],\mathbb{E}[g]\} \geq C_0 q^{-(m+n-s-1)(s+1)/2}$, $k=12s+12$, $n$ is sufficiently large depending on $s$, and $C_0$ is sufficiently large, a contradiction. (To check the last inequality, note that we clearly have
$$O(1)q^{-(m+n-s-1)(s+1)/2}\sqrt{\mathbb{E}[f]\mathbb{E}[g]} < \tfrac{1}{2}\mathbb{E}[f]\mathbb{E}[g]$$
provided $C_0$ is sufficiently large, so it suffices to check that
$$O(1)\sum_{d=1}^{s} q^{k^2d^2} q^{-(m+n-d)d/2} q^{3nd/4}C(\mathbb{E}[f] \mathbb{E}[g])^{1-1/(2k)} < \tfrac{1}{2}\mathbb{E}[f]\mathbb{E}[g],$$
or, rearranging and substituting the values $k=12s+12$ and $C=q^{m/16}$, that
$$O(1)\sum_{d=1}^{s} q^{(12s+12)^2d^2} q^{-(m+n-d)d/2} q^{3nd/4}q^{m/16} < \tfrac{1}{2}(\mathbb{E}[f]\mathbb{E}[g])^{1/(24s+24)}.$$
By our assumed lower bound on $\min\{\mathbb{E}[f],\mathbb{E}[g]\}$, and the fact that $C_0 \geq 1$, this will hold provided
$$O(1)\sum_{d=1}^{s} q^{(12s+12)^2d^2} q^{-(m+n-d)d/2} q^{3nd/4}q^{m/16} < \tfrac{1}{2}q^{-(m+n-s-1)/24}.$$
By our assumption $n \geq m \geq 3n/4$, we have $-(m+n)/2+3n/4 \leq -n/8$, and therefore the summand $$q^{(12s+12)^2d^2} q^{-(m+n-d)d/2} q^{3nd/4}q^{m/16}$$
decreases exponentially in $d$ (with ratio at most $1/2$), provided $n$ is sufficiently large depending on $s$. Hence, it suffices to consider the $d=1$ term --- more precisely, to check that
$$O(1)q^{(12s+12)^2} q^{-(m+n-1)/2} q^{3n/4}q^{m/16} <\tfrac{1}{4} q^{-(m+n-s-1)/24}.$$
This indeed holds for $n$ sufficiently large depending on $s$, since $-(m+n)/2+3n/4+m/16 < -(m+n)/24 - \Omega(n)$.)
\end{proof}

The following is the lemma we actually need; it is proved in a similar way.

\begin{lem}
\label{lem:hoffman-gen}
Let $s,t \in \mathbb{N}$. Suppose that $f,g:\Lin(V,W) \to [0,1]$ are such that $f(\sigma_1)g(\sigma_2) = 0$ whenever $\dim(\mathfrak{a}(\sigma_1,\sigma_2)) = t-1$, and that $f$ and $g$ are $(s,q^{m/16})$-quasiregular, where $\dim(V)=m$, $\dim(W)=n$ and $m-t+1 \leq n \leq(1+\tfrac{1}{8t})m$. Then for $n$ sufficiently large depending on $s$ and $t$ we have
$$\min\{\mathbb{E}[f],\mathbb{E}[g]\} = O(1) q^{-(m+n-s-1)(s+1)/2 + (n-m)(t-1)/2 + t^2/2}.$$
\end{lem}
\begin{proof}
Suppose that $f,g:\Lin(V,W) \to [0,1]$ are such that $f(\sigma_1)g(\sigma_2) = 0$ whenever $\dim(\mathfrak{a}(\sigma_1,\sigma_2)) = t-1$, and that $f$ and $g$ are $(s,C)$-quasiregular, where $\dim(V)=m$, $\dim(W)=n$, $n \geq m-t+1$ and $C := q^{m/16}$. Suppose for a contradiction that $\min\{\mathbb{E}[f],\mathbb{E}[g]\} \geq C_0 q^{-(m+n-s-1)(s+1)/2+(n-m)(t-1)/2+t^2/2}$, where $C_0 \geq 1$ is an absolute constant to be chosen later. By Lemma \ref{lem:eval-bound-2}, for any $s \in \mathbb{N}$ we have
$$|\lambda_d^{(t-1)}| \leq O(q^{-((m+n-s-1)(s+1)-(n-m)(t-1)-t^2)/2})\quad \forall s+1 \leq d \leq \min\{m,n\},$$
using the fact that $d \mapsto (m+n-d)d$ is an increasing function, for $0 \leq d \leq (m+n)/2$. Applying (\ref{eq:level-d}) and the Cauchy-Schwarz inequality, we have
$$|\langle f^{(=d)},g^{(=d)}\rangle| \leq \|f^{(=d)}\|_2 \|g^{(=d)}\|_2 \leq O(1)q^{k^2d^2} q^{3\max\{m,n\}d/4}C(\mathbb{E}[f] \mathbb{E}[g])^{1-1/(2k)}.$$
Hence, we have
\begin{align*} 0& =\langle f,M_{t-1}g \rangle\\
 &= \mathbb{E}[f] \mathbb{E}[g] + \sum_{d=1}^{n} \lambda_d^{(t-1)} \langle f^{(=d)},g^{(=d)}\rangle \\
& = \mathbb{E}[f]\mathbb{E}[g] + \sum_{d=1}^{s} \lambda_d^{(t-1)} \langle f^{(=d)},g^{(=d)}\rangle - O(1)q^{-(m+n-s-1)(s+1)/2 + (n-m)(t-1)/2 +t^2/2}\sqrt{\mathbb{E}[f]\mathbb{E}[g]}\\
& \geq \mathbb{E}[f]\mathbb{E}[g] - O(1) \sum_{d=1}^{s}  q^{-(m+n-d)d/2 + (n-m)(t-1)/2 + t^2/2} q^{k^2d^2} q^{3\max\{m,n\}d/4}C(\mathbb{E}[f] \mathbb{E}[g])^{1-1/(2k)}\\
& - O(1)q^{-(m+n-s-1)(s+1)/2 + (n-m)(t-1)/2 +t^2/2}\sqrt{\mathbb{E}[f]\mathbb{E}[g]}\\
& >0,
\end{align*}
provided $\min\{\mathbb{E}[f],\mathbb{E}[g]\} \geq C_0 q^{-(m+n-s-1)(s+1)/2+(n-m)(t-1)/2+t^2/2}$, $k = 12s+12$, $n \leq (1+\tfrac{1}{8t})m$, $C_0$ is sufficiently large, and $n$ is sufficiently large depending on $s$ and $t$, a contradiction. (To check the last inequality, we may proceed similarly to in the proof of the previous lemma. Clearly, we have
$$O(1)q^{-(m+n-s-1)(s+1)/2 + (n-m)(t-1)/2 +t^2/2}\sqrt{\mathbb{E}[f]\mathbb{E}[g]} < \tfrac{1}{2}\mathbb{E}[f] \mathbb{E}[g]$$
provided $C_0$ is sufficiently large. So it suffices to check that
$$O(1) \sum_{d=1}^{s}  q^{-(m+n-d)d/2 + (n-m)(t-1)/2 + t^2/2} q^{k^2d^2} q^{3\max\{m,n\}d/4}C(\mathbb{E}[f] \mathbb{E}[g])^{1-1/(2k)} < \tfrac{1}{2}\mathbb{E}[f] \mathbb{E}[g],$$
or, rearranging and substituting the values $k=12s+12$ and $C=q^{m/16}$, that
$$O(1) \sum_{d=1}^{s}  q^{-(m+n-d)d/2 + (n-m)(t-1)/2 + t^2/2} q^{(12s+12)^2d^2} q^{3\max\{m,n\}d/4}q^{m/16} < \tfrac{1}{2}(\mathbb{E}[f] \mathbb{E}[g])^{1/(24s+24)}.$$
By our assumed lower bound on $\min\{\mathbb{E}[f],\mathbb{E}[g]\}$, and the fact that $C_0 \geq 1$, this will hold provided
\begin{align*}&O(1) \sum_{d=1}^{s}  q^{-(m+n-d)d/2 + (n-m)(t-1)/2 + t^2/2} q^{(12s+12)^2d^2} q^{3\max\{m,n\}d/4}q^{m/16} \\
&< \tfrac{1}{2}q^{-(m+n-s-1)/24+(n-m)(t-1)/(24s+24)+t^2/(24s+24)}.\end{align*}
By our assumption $m-t+1 \leq n \leq(1+\tfrac{1}{8t})m$, we have
$$-(m+n)/2+3\max\{m,n\}/4 \leq -5m/32+(t-1)/2,$$ and therefore the summand
$$q^{-(m+n-d)d/2 + (n-m)(t-1)/2 + t^2/2} q^{(12s+12)^2d^2} q^{3\max\{m,n\}d/4}q^{m/16}$$
decreases exponentially in $d$ (with ratio at most $1/2$), provided $n$ is sufficiently large depending on $s$ and $t$. Hence, it suffices to consider the $d=1$ term --- more precisely, to check that
\begin{align*}&q^{-(m+n-1)/2 + (n-m)(t-1)/2 + t^2/2} q^{(12s+12)^2} q^{3\max\{m,n\}/4}q^{m/16}\\
&< \tfrac{1}{4} q^{-(m+n-s-1)/24+(n-m)(t-1)/(24s+24)+t^2/(24s+24)}.\end{align*}
This indeed holds for $n$ sufficiently large depending on $s$ and $t$, since
$-(m+n)/2 + (n-m)(t-1)/2 + 3\max\{m,n\}/4+m/16 < -(m+n)/24 - \Omega(n)$.)
\end{proof}

\section{Proof of Theorem \ref{theorem:junta-approximation}.}

It suffices to prove that for $m_0 \leq m \leq n \leq (1+1/(9t))m$ and $m_0 = m_0(r,t)$ large enough depending on $r$ and $t$, if $\f \subset \Lin(V,W)$ is $(t-1)$-intersection-free, then the junta 
$$\J = \langle (\Pi_1,\pi_1),(\Pi_2,\pi_2),\ldots(\Pi_N,\pi_N)\rangle \subset \Lin(V,W)$$
supplied by our (regularity) Lemma \ref{lem:reg-lem} (applied to $\f$, with $s=2145r^3$) is strongly $t$-intersecting: in other words, that for any $i,j \in [N]$, either $\dim(\mathfrak{a}(\Pi_i,\Pi_j)) \geq t$ or else $\dim(\mathfrak{a}(\pi_i,\pi_j)) \geq t$. So assume for a contradiction that $\dim(\mathfrak{a}(\Pi_1,\Pi_2)) = d < t$ and that $\dim(\mathfrak{a}(\pi_1,\pi_2))=d'< t$.

The regularity lemma implies that $\dim(\Domain(\Pi_i))+ \dim(\Domain(\pi_i)) \leq r$ for $i=1,2$, and that $\f(\Pi_1,\pi_1)$ and $\f(\Pi_2,\pi_2)$ are both $(s,q^{-mr +r^2/4})$-uncaptureable. By assumption, they are $(t-1)$-cross-intersection-free, when viewed as subsets of $\Lin(V,W)$. We will first use the uncaptureability property to find large subsets of $\f(\Pi_1,\pi_1)$ and $\f(\Pi_2,\pi_2)$ that are highly quasiregular; this will (eventually) contradict the $(t-1)$-cross-intersection-free property. If $\Pi_1$ and $\Pi_2$ had the same domain and the same range and the same were true of $\pi_1$ and $\pi_2$, the proof would be very short (given what has come before); a lot of the technicality below is to reduce to this case, by passing to dense subsets (and transforming appropriately) in such a way as to preserve the relevant intersection properties.

Write $E := \mathfrak{a}(\Pi_1,\Pi_2)$ and $C:= \mathfrak{a}(\pi_1,\pi_2)$. Let $D_i = \Domain(\Pi_i)$ for each $i \in \{1,2\}$, and let $B_i = \Domain(\pi_i)$ for each $i \in \{1,2\}$. Let $E_1 \leq D_1$ such that $E \oplus E_1 = D_1$, let $E_2 \leq D_2$ such that $E \oplus E_2 = D_2$, let $C_1 \leq B_1$ such that $C \oplus C_1 = B_1$ and let $C_2 \le B_2$ such that $C \oplus C_2 = B_2$. Let $\Lambda_1 = \Pi_1 \mid_{E_1}$, let $\Lambda_2 = \Pi_2 \mid_{E_2}$, let $\lambda_1 = \pi_1 \mid_{C_1}$ and let $\lambda_2 = \pi_2 \mid_{C_2}$.

Let $\alpha : = q^{m/32}$, let $s' := 67r^2$ and let $s'':=2r$. Since $\f(\Pi_1,\pi_1)$ is $(s,q^{-mr +r^2/4})$-uncaptureable, we have $\mu(\f(\Pi_1,\pi_1)) > q^{-mr}$. We now perform a process as follows. If $\f(\Pi_1,\pi_1)$ is $(s',\alpha)$-quasiregular, then stop. If not, there exist subspaces $S_1$ of $V$ and $A_1$ of $W^*$ and linear maps $\Psi_1 \in \Lin(S_1,W)$, $\psi_1 \in \Lin(A_1,V^*)$ such that $\dim(S_1)+\dim(A_1) \leq s'$, $S_1 \cap D_1 = \{0\}$, $A_1 \cap B_1 = \{0\}$ and $\mu^{|\cdots}(\f(\Pi_1,\pi_1,\Psi_1,\psi_1)) > \alpha \mu^{|\cdots}(\f(\Pi_1,\pi_1))$. If $\f(\Pi_1,\pi_1,\Psi_1,\psi_1)$ is $(s',\alpha)$-quasiregular, then stop. If not, there exist subspaces $S_2$ of $V$ and $A_2$ of $W^*$ and linear maps $\Psi_2 \in \Lin(S_2,W)$, $\psi_2 \in \Lin(A_2,V^*)$ such that $\dim(S_2)+\dim(A_2) \leq s'$, $S_2 \cap (D_1+S_1) = \{0\}$, $A_2 \cap (B_2+A_1) = \{0\}$ and $\mu^{|\cdots}(\f(\Pi_1,\pi_1,\Psi_1,\psi_1,\Psi_2,\psi_2)) > \alpha \mu^{|\cdots}(\f(\Pi_1,\pi_1,\Psi_1,\psi_1))$. Continue. This process must terminate after at most $32r$ steps. When it terminates, after $L \leq 32r$ steps, say, we have a family
$$\f(\Pi_1,\pi_1,\Psi_1,\psi_1,\Psi_2,\psi_2,\ldots,\Psi_L,\psi_L)$$
which is $(s',\alpha)$-quasiregular and has measure greater than $q^{-mr}$. For brevity, we write $\Psi = \sum_{i=1}^{L} \Psi_i$ and $\psi = \sum_{i=1}^{L} \psi_i$, so that
$$\f(\Pi_1,\pi_1,\Psi_1,\psi_1,\Psi_2,\psi_2,\ldots,\Psi_L,\psi_L) = \f(\Pi_1,\pi_1,\Psi,\psi),$$
and $|\Domain(\Psi)|+|\Domain(\psi)| \leq 32rs'$. Since $s\geq 32rs'+r$ and $\f(\Pi_2,\pi_2)$ is $(s,q^{-mr})$-uncaptureable, we may choose linear maps $\Theta_2,\theta_2$ with $\Domain(\Theta_2) = \Domain(\Lambda_1)\oplus \Domain(\Psi)$ and $\Domain(\theta_2) = \Domain(\lambda_1) \oplus \Domain(\psi)$, with $\mathfrak{a}(\Theta_2,\Lambda_1+\Psi) = \{0\}$ and $\mathfrak{a}(\theta_2,\lambda_1+\psi) = \{0\}$, and with
$$\mu^{|\cdots}(\f(\Pi_2,\pi_2,\Theta_2,\theta_2)) > q^{-mr}.$$
We now perform the following process. If $\f(\Pi_2,\pi_2,\Theta_2,\theta_2)$ is $(s'',\alpha)$-quasiregular, then stop. If not, there exist subspaces $T_1$ of $V$ and $F_1$ of $W^*$ and linear maps $\Phi_1 \in \Lin(T_1,W)$, $\phi_1 \in \Lin(F_1,V^*)$ such that $\dim(T_1)+\dim(F_1) \leq s''$, $T_1 \cap (\Domain(\Pi_2)+\Domain(\Theta_2)) = \{0\}$, $F_1 \cap (\Domain(\pi_2)+\Domain(\theta_2)) = \{0\}$ and
$$\mu^{|\cdots}(\f(\Pi_2,\pi_2,\Theta_2,\theta_2,\Phi_1,\phi_1)) > \alpha \mu^{|\cdots}(\f(\Pi_2,\pi_2,\Theta_2,\theta_2)).$$
If $\f(\Pi_2,\pi_2,\Theta_2,\theta_2,\Phi_1,\phi_1)$ is $(s'',\alpha)$-quasiregular, then stop. If not, there exist subspaces $T_2$ of $V$ and $F_2$ of $W^*$ and linear maps $\Phi_2 \in \Lin(T_2,W)$, $\phi_2 \in \Lin(F_2,V^*)$ such that $\dim(T_2)+\dim(F_2) \leq s''$, $T_2 \cap (\Domain(\Pi_2)+\Domain(\Theta_2)+T_1) = \{0\}$, $F_2 \cap (\Domain(\pi_2)+\Domain(\theta_2)+F_2) = \{0\}$ and
$$\mu^{|\cdots}(\f(\Pi_2,\pi_2,\Theta_2,\theta_2,\Phi_1,\phi_1,\Phi_2,\phi_2)) > \alpha \mu^{|\cdots}(\f(\Pi_2,\pi_2,\Theta_2,\theta_2,\Phi_1,\phi_1)).$$
Continue. This process must terminate after at most $32r$ steps. When it terminates, after $M \leq 32r$ steps, say, we have a family
$$\f(\Pi_2,\pi_2,\Theta_2,\theta_2,\Phi_1,\phi_1,\Phi_2,\phi_2,\ldots,\Phi_M,\phi_M)$$
which is $(s'',\alpha)$-quasiregular and has measure greater than $q^{-mr}$. For brevity, we write $\Phi = \sum_{i=1}^{M} \Phi_i$ and $\phi = \sum_{i=1}^{M} \phi_i$, so that
$$\f(\Pi_2,\pi_2,\Theta_2,\theta_2,\Phi_1,\phi_1,\Phi_2,\phi_2,\ldots,\Phi_M,\phi_M) = \f(\Pi_2,\pi_2,\Theta_2,\theta_2,\Phi,\phi),$$
and $\Domain(\Phi)+\Domain(\phi) \leq 32rs''$.

Write $\delta: = \mu(\f(\Pi_1,\pi_1,\Psi,\psi))$. Since $\f(\Pi_1,\pi_1,\Psi,\psi)$ is $(1,\alpha)$-quasiregular, it is $(m/2,\delta/2)$-uncaptureable (by Claim \ref{claim:quasi-uncap}), so we may choose linear maps $\Theta_1,\theta_1$ with $\Domain(\Theta_1) = \Domain(\Lambda_2)\oplus \Domain(\Phi)$ and $\Domain(\theta_1) = \Domain(\lambda_2) \oplus \Domain(\phi)$, with $\mathfrak{a}(\Theta_1,\Lambda_2+\Phi) = \{0\}$ and $\mathfrak{a}(\theta_1,\lambda_2+\phi) = \{0\}$, and with $\mu(\f(\Pi_1,\pi_1,\Psi,\psi,\Theta_1,\theta_1)) > \delta/2$. (Simply average over all pairs of linear maps $(\Theta,\theta)$ such that $\Domain(\Theta) = \Domain(\Lambda_2)\oplus \Domain(\Phi)$, $\Domain(\theta) = \Domain(\lambda_2) \oplus \Domain(\phi)$, $\mathfrak{a}(\Theta,\Lambda_2+\Phi) = \{0\}$ and $\mathfrak{a}(\theta,\lambda_2+\phi) = \{0\}$.) Since $\f(\Pi_1,\pi_1,\Psi,\psi)$ is $(s',\alpha)$-quasiregular, it follows that $$\f(\Pi_1,\pi_1,\Psi,\psi,\Theta_1,\theta_1)$$
is $(s'-32rs''-r,2\alpha)$-quasiregular, and therefore $(s'',2\alpha)$-quasiregular, since $s'-32rs''-r \geq s''$. 

We now have a pair of families
$$\f(\Pi_1,\pi_1,\Psi,\psi,\Theta_1,\theta_1),\quad \f(\Pi_2,\pi_2,\Theta_2,\theta_2,\Phi,\phi)$$
that are both $(s'',2\alpha)$-quasiregular and of measure greater than $q^{-mr}/2$; moreover, by construction, we have
\begin{align*} \Domain(\Pi_1)\oplus\Domain(\Psi)\oplus\Domain(\Theta_1) & = \Domain(\Pi_2)\oplus\Domain(\Phi)\oplus\Domain(\Theta_2),\\
\Domain(\pi_1)\oplus\Domain(\psi)\oplus\Domain(\theta_1) &= \Domain(\pi_2)\oplus\Domain(\phi)\oplus\Domain(\theta_2),
\end{align*}
and
\begin{align*} \dim(\mathfrak{a}(\Pi_1+\Psi + \Theta_1,\Pi_2+\Phi+\Theta_2)) &= \dim(\mathfrak{a}(\Pi_1,\Pi_2))=d,\\
\dim(\mathfrak{a}(\pi_1+\psi + \theta_1,\pi_2+\phi+\theta_2)) &= \dim(\mathfrak{a}(\pi_1,\pi_2))=d'.\end{align*}

Write $\Pi = \Pi_1 \mid_{E} = \Pi_2 \mid_{E}$, and write $\pi = \pi_1 \mid_{C} = \pi_2 \mid_{C}$; write $\Xi_1 := \Pi_1\mid_{E_1} +\Psi + \Theta_1$, $\xi_1 := \pi_1\mid_{C_1} +\psi + \theta_1$, $\Xi_2 := \Pi_2\mid_{E_2} +\Phi + \Theta_2$ and $\xi_2 := \pi_2\mid_{C_2} +\phi + \theta_2$. Then the two families
$$\f(\Pi_1,\pi_1,\Psi,\psi,\Theta_1,\theta_1) = \f(\Pi,\pi,\Xi_1,\xi_1),\quad \f(\Pi_2,\pi_2,\Theta_2,\theta_2,\Phi,\phi) = \f(\Pi,\pi,\Xi_2,\xi_2)$$
are both $(s'',2\alpha)$-quasiregular and of measure greater than $q^{-mr}/2$. Moreover, we have $\dim(\Domain(\Pi))=d$, $\dim(\Domain(\pi))=d'$, $\Domain(\Xi_1) = \Domain(\Xi_2)$, $\Domain(\xi_1) = \Domain(\xi_2)$, $\mathfrak{a}(\Xi_1,\Xi_2) = \{0\}$ and $\mathfrak{a}(\xi_1,\xi_2) = \{0\}$. Our aim is to use these properties (including, crucially, the quasiregularity property) to find two linear maps
$$\sigma_1 \in \f(\Pi,\pi,\Xi_1,\xi_1),\quad \sigma_2 \in \f(\Pi,\pi,\Xi_2,\xi_2)$$
such that $\dim(\mathfrak{a}(\sigma_1,\sigma_2)) = t-1$. This will contradict our assumption that $\f$ is $(t-1)$-intersection-free.

It is now slightly clearer to adopt a matrix perspective. Let $k: = \dim(\Domain(\Xi_1)) = \dim(\Domain(\Xi_2))$ and let $l: = \dim(\Domain(\xi_1)) = \dim(\Domain(\xi_2))$. Note for later that $\max\{k,l\} \leq r+32rs' \leq 33rs'$. Choose a basis $\{v_1,\ldots,v_m\}$ for $V$ such that $\{v_{1},\ldots,v_d\}$ is a basis for $\Domain(\Pi)$ and $\{v_{d+1},\ldots,v_{d+k}\}$ is a basis for $\Domain(\Xi_1) = \Domain(\Xi_2)$, and choose a basis $\{w_1,\ldots,w_n\}$ for $W$ such that its dual basis $\{\eta_1,\ldots,\eta_{n}\}$ (for $W^*$) has the property that $\{\eta_1,\ldots,\eta_{d'}\}$ is a basis for $\Domain(\pi)$ and $\{\eta_{d'+1},\ldots,\eta_{d'+l}\}$ is a basis for $\Domain(\xi_1) = \Domain(\xi_2)$. With respect to such bases, the families of linear maps
$$\f(\Pi,\pi,\Xi_1,\xi_1),\quad \f(\Pi,\pi,\Xi_2,\xi_2)$$
correspond to families $\f_1,\f_2$ (respectively) of matrices in $\mathcal{M}(n,m)$ whose first $d+k$ columns and first $d'+l$ rows are fixed, with the first $d$ fixed columns being the same for $\f_1$ and $\f_2$ and the next $k$ fixed columns being different, and the first $d'$ fixed rows being the same for $\f_1$ and $\f_2$ and the next $l$ fixed rows being different. (The $(i,j)$-th entry of each matrix in $\f_h$ is equal to $\eta_i(\Pi(e_j))$ for each $i \in [n],\ j \in [d],\ h \in \{1,2\}$; the $(i,j)$-th entry of each matrix in $\f_h$ is equal to $(\pi(\eta_i))(v_j)$ for each $i \in [d'],\ j \in [m],\ h \in \{1,2\}$; the $(i,j)$-th entry of each matrix in $\f_h$ is equal to $\eta_i(\Xi_h(v_j))$ for each $i \in [n],\ j \in \{d+1,\ldots,d+k\},\ h \in \{1,2\}$; and the $(i,j)$-th entry of each matrix in $\f_h$ is equal to $(\xi_h(\eta_i))(v_j)$ for each $i \in \{d'+1,\ldots,d'+l\},\ j \in [m],\ h \in \{1,2\}$.) Of course, the uniform measure on $\Lin(V,W)(\Pi,\pi,\Xi_1,\xi_1)$ corresponds under this identification to the uniform measure on the copy of $\mathcal{M}(n-d'-l,m-d-k)$ produced by fixing the first $d+k$ rows and the first $d'+l$ columns as above, and similarly for the uniform measure on $\Lin(V,W)(\Pi,\pi,\Xi_2,\xi_2)$.

Let $A_0 \in \mathcal{M}(n,m)$ be a matrix whose first $d+k$ columns are the same as the first $d+k$ columns of any (all) matrices in $\f_1$, and whose first $d'+l$ rows are the same as the first $d'+l$ rows of any (all) matrices in $\f_1$, and whose other entries are all zero. Clearly, for any $v \in V$ and any matrices $A_1,A_2 \in \mathcal{M}(n,m)$, we have $(A_1-A_0)(v) = (A_2-A_0)(v)$ if and only if $A_1 v = A_2 v$, so by replacing $\f_1$ and $\f_2$ by the translated families $\f_1-A_0$ and $\f_2-A_0$ if necessary, we may assume that the first $d+k$ columns and the first $d'+l$ rows of all matrices in $\f_1$ are zero, and therefore the first $d$ columns and the first $d'$ rows of all matrices in $\f_2$ are all zero. (Note that this translation preserves the measures and the quasiregularity of the two families.) Now let $\iota:\mathcal{M}(n,m) \to \mathcal{M}(n-d',m-d)$ be the map under which the first $d$ columns and the first $d'$ rows of an $n$ by $m$ matrix are deleted. It is easy to check that $\dim(\mathfrak{a}(A_1,A_2)) = \dim(\mathfrak{a}(\iota(A_1),\iota(A_2))) + d$ for any matrices $A_1 \in \f_1$ and $A_2 \in \f_2$, using the fact that the first $d$ columns and the first $d'$ rows of every matrix in $\f_1$ or $\f_2$ are all zero. Hence, defining $\g_1 := \iota(\f_1) \subset \mathcal{M}(n-d',m-d)$ and $\g_2 := \iota(\f_2) \subset \mathcal{M}(n-d',m-d)$, it suffices to find two matrices $B_1 \in \g_1$ and $B_2 \in \g_2$ such that $\dim(\mathfrak{a}(B_1,B_2))= t-1-d$. Note that the first $k$ columns and the first $l$ rows of every matrix in $\g_1$ are all zero, and the first $k$ columns and the first $l$ rows of any two matrices in $\g_2$ agree with one another.

Let $c_1,\ldots,c_k$ be the first $k$ columns (in order) of any (every) matrix in $\g_2$, and let $r_1,\ldots,r_l$ be the first $l$ rows (in order) of any (every) matrix in $\g_2$. Let $B_0$ be the matrix whose first $k$ columns (in order) are $c_1,\ldots,c_k$ and whose first $l$ rows (in order) are $r_1,\ldots,r_l$, and in whose other entries we place the symbol $*$ (denoting an indeterminate); note that $B_0$ records the fixed entries of matrices in $\g_2$. Since $\mathfrak{a}(\Xi_1,\Xi_2) = \{0\}$, the first $k$ columns of $B_0$ are linearly independent, and since $\mathfrak{a}(\xi_1,\xi_2) = \{0\}$, the first $l$ rows of $B_0$ are linearly independent. Let $E_0$ denote the top-left $l$ by $k$ minor of $B_0$, i.e.\ the matrix consisting of the intersection of the first $k$ (fixed) columns $c_1,\ldots,c_k$ and the first $l$ (fixed) rows $r_1,\ldots,r_l$. By a suitable change of bases (which corresponds to performing elementary row operations involving only the fixed rows $r_1,\ldots,r_l$ and elementary column operations involving only the fixed columns $c_1,\ldots,c_k$), we may assume that $E_0$ contains a $u$ by $u$ identity matrix in its top left corner, and has all its other entries equal to zero. Then the $l-u$ by $m-k-d$ submatrix $D_0$ of $B_0$ formed by intersecting the last $m-k-d$ columns of $B_0$ with the rows $r_{u+1},\ldots,r_l$, has linearly independent rows, and similarly the $n-d'-l$ by $k-u$ submatrix $F_0$ of $B_0$ formed by intersecting the last $n-d'-l$ rows of $B_0$ with the columns $c_{u+1},\ldots,c_{k}$, has linearly independent columns. Let $D_0'$ denote the submatrix of $B_0$ formed by intersecting the last $m-d-k$ columns of $B_0$ with the first $u$ rows of $B_0$, and let $F_0'$ denote the submatrix of $B_0$ formed by intersecting the last $n-d'-l$ rows of $B_0$ with its first $u$ columns. Schematically, we have
$$ B_0 = \begin{pmatrix*}[l]
            I_{u \times u} & O_{u \times (k-u)} & (D_0')_{u \times (m-d-k)} \\
            O_{(l-u) \times u} & O_{(l-u) \times (k-u)} & (D_0)_{(l-u) \times (m-d-k)} \\
            (F_0')_{(n-d'-l) \times u} & (F_0)_{(n-d'-l) \times (k-u)} & (*)_{(n-d'-l) \times (m-d-k)}
        \end{pmatrix*},
$$
where $O$ denotes the all-zeros matrix, $(*)_{p \times q}$ denotes matrix with $p$ rows and $q$ columns and with the symbol $*$ in every entry, and more generally the subscripts denote the dimensions of the relevant submatrices (number of rows followed by number of columns, as usual).

For each matrix $A_1 \in \g_1$, let $A_1'$ be the matrix produced by deleting its first $k$ (fixed, zero-valued) columns and its first $l$ (fixed, zero-valued) rows; similarly, for each $A_2 \in \g_2$, let $A_2'$ be the matrix produced by deleting its first $k$ (fixed) columns and its first $l$ (fixed) rows. Let $\g_h': = \{A_h':\ A_h \in \g_h\}$ (for $h=1,2$) denote the corresponding families of matrices. Then, writing a vector $v \in \mathbb{F}_q^{m-d}$ in the form
$$\begin{pmatrix}x \\ y\\ z\end{pmatrix}$$
where $x \in \mathbb{F}_q^{u}$, $y \in \mathbb{F}_q^{k-u}$, $z \in \mathbb{F}_q^{m-d-k}$, provided $A_i \in \g_i$ for $i=1,2$ we have $A_2 v = A_1 v$ if and only if the following system of linear equations is satisfied:
\begin{align}
x+D_0'z & = 0, \label{eq:simult-1}\\
D_0z &= 0, \label{eq:simult-2}\\
F_0'x+F_0y+A_2'z &= A_1'z. \label{eq:simult-3}
\end{align}
Substituting (\ref{eq:simult-1}) into (\ref{eq:simult-3}) and rearranging, we have the following equivalent system:
\begin{align}
x& = -D_0'z, \label{eq:simult-4}\\
D_0z &= 0, \label{eq:simult-5}\\
(A_1' - A_2' + F_0'D_0')z &= F_0 y. \label{eq:simult-6}
\end{align}
Given $z \in \ker(D_0)$ such that $(A_1' - A_2' + F_0'D_0')z$ is contained in the column space of $F_0$, there is a unique solution $(x,y,z)$ to the above system; conversely, for any solution $(x,y,z)$ to the above system, we must have $z \in \ker(D_0)$ and $(A_1' - A_2' + F_0'D_0')z \in \text{columnspace}(F_0)$. It follows that
$$\dim(\mathfrak{a}(A_1,A_2)) = \dim\{z \in \ker(D_0):\ (A_1' - A_2' + F_0'D_0')z \in \text{columnspace}(F_0)\}.$$
Since $D_0$ has linearly independent rows, we have $\dim(\ker(D_0)) = (m-d-k)-(l-u) = m-d-k-l+u$; since $F_0$ has linearly independent columns, the dimension of its column space is $k-u$. We now define a (linear) map
$$\Gamma: \mathcal{M}(n-d'-l,m-d-k)\to  \Lin(\ker(D_0),\mathbb{F}_q^{n-d'-l}/\text{columnspace}(F_0))$$
as follows: for a matrix $A \in \mathcal{M}(n-d'-l,m-d-k)$, we define $\Gamma(\sigma) \in \Lin(\ker(D_0),\mathbb{F}_q^{n-d'-l}/\text{columnspace}(F_0))$ to be the linear map produced by restricting $A$ to $\ker(D_0)$ and composing this restriction with the natural quotient map from $\mathbb{F}_q^{n-d'-l}$ to $\mathbb{F}_q^{n-d'-l}/\text{columnspace}(F_0)$. For brevity, we write $\tilde{V} := \ker(D_0)$ and $\tilde{W} := \mathbb{F}_q^{n-d'-l}/\text{columnspace}(F_0)$; we also write $\tilde{m}:= \dim(\tilde{V}) = m-d-k-l+u$ and $\tilde{n}: = \dim(\tilde{W}) = n-d'-l-(k-u) = n-d'-l-k+u$; note that $\tilde{n} - \tilde{m} = n-m+d-d'$. 

We now define functions $f_1:\Lin(\tilde{V},\tilde{W}) \to [0,1]$ and $f_2:\Lin(\tilde{V},\tilde{W}) \to [0,1]$ by
$$f_1(\sigma) = \frac{|\Gamma^{-1}(\sigma) \cap \g_1'|}{|\Gamma^{-1}(\sigma)|} \quad \forall \sigma \in \Lin(\tilde{V},\tilde{W})$$
and
$$f_2(\sigma) = \frac{|\Gamma^{-1}(\sigma) \cap (\g_2' - F_0'D_0')|}{|\Gamma^{-1}(\sigma)|} \quad \forall \sigma \in \Lin(\tilde{V},\tilde{W}).$$
It is easy to see that $\mathbb{E}[f_h] > q^{-mr}/2$ for $h=1,2$, and that each $f_h$ is $(s'',2\alpha)$-quasiregular; moreover, if there exist $\sigma_1,\sigma_2 \in \Lin(\tilde{V},\tilde{W})$ such that $f_h(\sigma_h)>0$ for $h=1,2$ and $\dim(\mathfrak{a}(\sigma_1,\sigma_2)) = t-1-d$, then there exist $B_1 \in \g_1$ and $B_2 \in \g_2$ such that $\dim(\mathfrak{a}(B_1,B_2)) = t-1-d$. To obtain our desired contradiction, it therefore suffices to show that there exist $\sigma_1,\sigma_2 \in \Lin(\tilde{V},\tilde{W})$ such that $f_h(\sigma_h)>0$ for $h=1,2$ and $\dim(\mathfrak{a}(\sigma_1,\sigma_2)) = t-1-d$. This is an immediate consequence of Lemma \ref{lem:hoffman-gen}, applied with $\tilde{t}:=t-d$ in place of $t$, $\tilde{n}$ in place of $n$, $\tilde{m}$ in place of $m$ and $\tilde{s}: = s''$ in place of $s$; note that $\tilde{n}-\tilde{m}+\tilde{t}-1 = n-m+d-d'+(t-d)-1 = n-m+t-1-d' \geq n-m \geq 0$, since by hypothesis, $d' \leq t-1$ and $n \geq m$. Indeed, if $f_1(\sigma_1)f_2(\sigma_2) = 0$ whenever $\sigma_1,\sigma_2 \in \Lin(\tilde{V},\tilde{W})$ are such that $\dim(\mathfrak{a}(\sigma_1,\sigma_2))=\tilde{t}-1$, then Lemma \ref{lem:hoffman-gen} yields
$$q^{-mr}/2 < \min\{\mathbb{E}[f_1],\mathbb{E}[f_2]\} < O(1) q^{-(\tilde{m}+\tilde{n}-\tilde{s}-1)(\tilde{s}+1)/2+(\tilde{n}-\tilde{m})(\tilde{t}-1)/2+\tilde{t}^2/2},$$
a contradiction provided $m_0 \leq m \leq n \leq (1+1/(9t))m$ and $m_0 = m_0(r,t)$ is chosen to be sufficiently large depending on $t$ and $r$. (Note that here, we used the facts that $\tilde{s} = 2r \geq r+t$, $\tilde{n} =n-d'-k-l+u \geq n-t+1-66rs'$, $\tilde{m} = m-d-k-l+u \geq m-t+1-66rs'$, and $\tilde{n}-\tilde{m} = n-m+d-d' \leq n-m+t-1 \leq \tilde{m}/(8\tilde{t})$.)

\section{Extremal results}
Theorem \ref{theorem:junta-approximation} straightforwardly implies Theorem \ref{theorem:glnfq}; in this section, we give the deduction. First, two straightforward linear algebraic lemmas are helpful.

\begin{lem}
    Let $q$ be a prime power, let $V$ be an $n$-dimensional vector space over $\mathbb{F}_q$, let $U$ be a $k$-dimensional subspace of $V$, and suppose $d \in \mathbb{N}$ with $k+d \leq n$. Then the number of $d$-dimensional subspaces $W$ of $V$ such that $W \cap U = \{0\}$ is at least
    $$\tfrac{1}{4}{n \brack d}_q.$$
\end{lem}
\begin{proof}
    The number of choices for an ordered, linearly independent set of size $d$ whose span $S$ satisfies $S \cap U = \{0\}$ is clearly equal to 
    $$\prod_{i=1}^{d}(q^{n}-q^{k+i-1}),$$
    and for a given $d$-dimensional subspace $S$, the number of choices of an ordered basis of $S$ is equal to 
    $$\prod_{i=1}^{d}(q^{d}-q^{i-1}),$$
    so the number of $d$-dimensional subspaces $W$ of $U$ such that $W \cap U = \{0\}$ is equal to
    \begin{align*}
        \frac{\prod_{i=1}^{d}(q^{n}-q^{k+i-1})}{\prod_{i=1}^{d}(q^{d}-q^{i-1})} & = \frac{\prod_{i=1}^{d}(q^{n-i+1}-q^{k})}{\prod_{i=1}^{d}(q^{d-i+1}-1)}\\
        & = \frac{\prod_{i=1}^{d}(q^{n-i+1}-q^{k})}{\prod_{i=1}^{d}(q^{n-i+1}-1)} {n \brack d}_q\\
        & \geq {n \brack d}_q\prod_{i=1}^{d}(1-q^{-(n-k-i+1)}) \\
        & \geq   {n \brack d}_q \prod_{i=1}^{n-k} (1-q^{-(n-k-i+1)}) \\
        & =  {n \brack d}_q \phi(n-k,n-k) \\
        & > \tfrac{1}{4} {n \brack d}_q, 
        \end{align*}
        using the bound on $\phi(\cdot,\cdot)$ from the proof of Lemma \ref{lem:eval-bound-1}.
    \end{proof}
\begin{lem}
\label{lem:derangements}
Let $q$ be a prime power, let $V$ be an $n$-dimensional vector space over $\mathbb{F}_q$, let $t \in \mathbb{N}$, let $$\J = \{\sigma \in \Lin(V,V): \sigma(e_i) = e_i\ \forall i \in [t]\}$$
be the family of all linear maps fixing the first $t$ standard basis vectors, assume $3t \leq n$, and let
$$m_{q,t}(n): = \prod_{i=1}^{n-t}(q^n - q^{i+t-1})= |\J \cap \GL(V)|.$$
Let $\tau \in \GL(V)$ such that $\dim(\{v \in \Span\{e_1,\ldots,e_t\}:\ \tau(v)=v\}) \leq t-1$. Let $\mathcal{H}: = \{\sigma \in \J \cap \GL(V):\ \dim(\mathfrak{a}(\sigma,\tau)) = t-1\}$. Then $|\mathcal{H}| = \Omega(q^{-(t-1)^2}m_{q,t}(n))$.    
\end{lem}
\begin{proof}
    The case $t=1$ follows easily from the (known) asymptotic for the number of linear derangements in $\GL(\mathbb{F}_q^n)$ (see \cite{morrison}), so assume henceforth that $t \geq 2$. Let $T = \Span\{e_1,\ldots,e_t\}$, let $D = \{v \in T:\ \tau(v)=v\}$ and let $d = \dim(D)$; note that $0\leq d \leq t-1$ and that $\dim(T + \tau^{-1}(T)) \leq 2t$. We may construct an element $\sigma \in \mathcal{H}$ by the following process. First, choose a $(t-d-1)$-dimensional subspace $W$ of $V$ such that $W \cap (T + \tau^{-1}(T)) = \{0\}$ and define $\sigma(v)=v$ for all $v \in T$ and $\sigma(v)=\tau(v)$ for all $v \in W$. Now fix an ordered basis $v_1,v_2,\ldots,v_n$ for $V$ such that $v_i = e_i$ for $1 \leq i \leq t$ and such that $v_1,v_2,\ldots,v_{2t-d-1}$ is an ordered basis for $T+W$. Then, for each $i=2t-d,\ldots,n$ in turn, choose $\sigma(v_i)$ subject to the conditions (i) $\sigma(v_i) \notin \Span\{\sigma(v_1),\ldots,\sigma(v_{i-1})\}$ and (ii) $\sigma(v_i) \notin \tau(v_i)+\Span\{\sigma(v_j)-\tau(v_j): j < i\}$. We claim that this process does indeed produce an element of $\mathcal{H}$:
    
\begin{claim}
    $\sigma \in \mathcal{H}$.
\end{claim}
\begin{proof}[Proof of Claim.]
Observe firstly that $\sigma \in \J$, since $\sigma$ fixes every vector in $T$. Secondly, the condition $W \cap \tau^{-1}(T) = \{0\}$ (or equivalently $\tau(W) \cap T = \{0\}$) guarantees that $\sigma|_{W+T}$ has full rank; condition (i) then guarantees that $\sigma$ is invertible. Finally, we will check that $\{v \in V: \sigma(v)=\tau(v)\} = D+W$; since $\dim(D+W)=\dim(D)+\dim(W) = t-1$, this will complete the proof of the claim. Clearly, $\sigma(v)=\tau(v)$ for all $v \in D+W$. Now let $v \in V$ with $\sigma(v)=\tau(v)$ and write $v = \sum_{j=1}^{n}\lambda_j v_j$. If $\lambda_i \neq 0$ for some $i \geq 2t-d$, then let $k$ be the maximal such $i$; we then have
    $$\sum_{j=1}^{k-1}\lambda_j \sigma(v_j)+\lambda_k\sigma(v_k) = \sigma(v)=\tau(v)=\sum_{j=1}^{k-1}\lambda_j \tau(v_j)+\lambda_k\tau(v_k),$$
    so $\sigma(v_k)-\tau(v_k) \in \Span\{\sigma(v_j)-\tau(v_j):\ j < k\}$, which contradicts condition (ii). It follows that, if $\sigma(v)=\tau(v)$, then $v \in W+T$; writing $v = w+u$ where $w \in W$ and $u \in T$, we then have
    $$\tau(w)+\tau(u) = \tau(v)=\sigma(v) = \sigma(w)+\sigma(u) = \tau(w)+u,$$
    so $\tau(u)=u$, thus $u \in D$, and therefore $v \in D+W$, as required. This completes the proof of the claim.
    \end{proof}
    
    We now bound from below the number of choices in the above process. Since $\dim(T+\tau^{-1}(T))+t-d-1 \leq 2t+t-d-1 \leq 3t \leq n$, we may apply the previous lemma to conclude that the number of choices for $W$ is at least 
    $$\tfrac{1}{4} {n \brack t-d-1}_q.$$
    Now consider the number of choices for $\sigma(v_i)$ for $i=2t-d,\ldots,n$. The subspace in condition (i) clearly has dimension $i-1$ and therefore cardinality $q^{i-1}$. Since $\sigma(v)-\tau(v) = 0$ for all $v\in D+W$, the affine subspace in condition (ii) has dimension at most $i-1-\dim(D+W) = i-1-(t-1) = i-t$, so cardinality at most $q^{i-t}$. Hence, the number of choices for $\sigma(v_i)$ is at least $q^n-q^{i-1}-q^{i-t}$, for each $i=2t-d,\ldots,n$. Putting all of this together, the total number of choices $\mathscr{N}$ for $\sigma$ satisfies
    $$ \mathscr{N} \geq \tfrac{1}{4} {n \brack t-d-1}_q \prod_{i=2t-d}^{n} (q^n - q^{i-1}-q^{i-t}).$$
    We have $|\mathcal{H}| \geq \mathscr{N}$, and therefore
    $$\frac{4|\mathcal{H}|}{m_{q,t}(n)} \geq \frac{{n \brack t-d-1}_q \prod_{i=2t-d}^{n} (q^n - q^{i-1}-q^{i-t})}{m_{q,t}(n)} = \frac{\prod_{i=1}^{t-d-1}(q^n-q^{i-1}) \prod_{i=2t-d}^{n} (q^n - q^{i-1}-q^{i-t})} {\prod_{i=1}^{t-d-1} (q^{t-d-1}-q^{i-1})\prod_{i=1}^{n-t}(q^n - q^{i+t-1})}.$$
   Trivially, $\prod_{i=1}^{t-d-1} (q^{t-d-1}-q^{i-1}) \leq \prod_{i=1}^{t-1} (q^{t-1}-q^{i-1}) \leq q^{(t-1)^2}$, so to complete the proof of the lemma it suffices to show that
    $$\frac{\prod_{i=1}^{t-d-1}(q^n-q^{i-1}) \prod_{i=2t-d}^{n} (q^n - q^{i-1}-q^{i-t})} {\prod_{i=1}^{n-t}(q^n - q^{i+t-1})} = \Omega(1).$$
    This indeed holds; we have
    \begin{align*} \frac{\prod_{i=1}^{t-d-1}(q^n-q^{i-1}) \prod_{i=2t-d}^{n} (q^n - q^{i-1}-q^{i-t})} {\prod_{i=1}^{n-t}(q^n - q^{i+t-1})} & \geq \prod_{j=1}^{n-2t+d+1}(1-q^{n-t-j+1}/(q^n-q^{n-j}))\\
    & \geq \prod_{j=1}^{n-2t+d+1}(1-q^{n-1-j}/(q^n-q^{n-j}))\\    & \geq \prod_{j=1}^{n-2t+d+1} (1-q^{-j})\\
    & = \phi(n-2t+d+1,n-2t+d+1)\\
    & > 1/4.
    \end{align*}
    
    \end{proof}

\begin{proof}[Proof of Theorem \ref{theorem:glnfq}.]
Let $n \in \mathbb{N}$ with $n \geq n_0$, where $n_0 = n_0(t)$ is to be chosen later, let $V$ be an $n$-dimensional vector space over $\mathbb{F}_q$, and let $\f \subset \GL(V)$ be $(t-1)$-intersection-free with
$$|\f| \geq \prod_{i=1}^{n-t}(q^n - q^{i+t-1}): = m_{q,t}(n).$$
Note that
$$q^{-n(n-t)}m_{q,t}(n) = q^{-n(n-t)} \prod_{i=1}^{n-t}(q^n - q^{i+t-1}) = \prod_{i=1}^{n-t}(1-q^{-(n-i-t+1)}) \geq \prod_{i=1}^{n-t}(1-2^{-i}) > \prod_{i=1}^{\infty}(1-2^{-i}) > \tfrac{1}{4},$$
so $q^{n(n-t)}/4 < m_{q,t}(n) < q^{n(n-t)}$. Applying Theorem \ref{theorem:junta-approximation} with $r=t+1$, we obtain a strongly $t$-intersecting $(C,r)$-junta $\J \subset \Lin(V,V)$ such that
\begin{equation}\label{eq:fminusjbound} |\f \setminus \J| \leq Cq^{-(t+1)n}|\Lin(V,V)|,
\end{equation}
where $C = q^{2146(t+1)^4}$. Suppose for a contradiction that $\J$ is not of the form $\langle \Pi \rangle$, where $\Pi \in \Lin(S,V)$ for some $S \leq V$ with $\dim(S)=t$, or of the form $\langle \pi \rangle$, where $\pi \in \Lin(A,V^*)$ for some $A \leq V^*$ with $\dim(A)=t$.  Write $\J = \langle (\Pi_1,\pi_1),(\Pi_2,\pi_2),\ldots(\Pi_N,\pi_N)\rangle$, where $N \leq C$. Since $\J$ is strongly $t$-intersecting, we either have $\J = \emptyset$ ($N=0$) or else
$$\dim(\Domain(\Pi_i))+\dim(\Domain(\pi_i)) \geq t+1 \quad \forall i \in [N],$$
and therefore
\begin{equation}\label{eq:j-bound} |\J| \leq Cq^{-((t+1)n - \lfloor (t+1)^2/4\rfloor)}|\Lin(V,V)| \leq q^{-(t+1)n+O(t^4)}|\Lin(V,V)|.\end{equation}
(Here, we have used the fact that, if $\dim(\Domain(\Pi_i))+\dim(\Domain(\pi_i)) \geq t+1$, then $|\langle(\Pi_i,\pi_i)\rangle| \leq q^{-(t+1)n + \lfloor (t+1)^2/4\rfloor} |\Lin(V,V)| = q^{n^2-(t+1)n + \lfloor (t+1)^2/4\rfloor}$. Indeed, if $\dim(\Domain(\Pi_i))=d$ and $\dim(\Domain(\pi_i))=d'$, then $|\langle(\Pi_i,\pi_i)\rangle| = q^{n^2-dn-d'n+dd'}$, the latter quantity being the number of possibilities for an $n$ by $n$ matrix over $\mathbb{F}_q$ with $d$ fixed columns and $d'$ fixed rows. Subject to the constraint $d+d'\geq t+1$, this latter quantity is clearly maximized by taking $\{d,d'\} = \{\lfloor (t+1)/2\rfloor, \lceil (t+1)/2\rceil\}$, yielding the above bound.) Combining (\ref{eq:j-bound}) with (\ref{eq:fminusjbound}) yields
$$|\f| \leq q^{-(t+1)n+O(t^4)}|\Lin(V,V)| \leq q^{n^2-(t+1)n+O(t^4)} < q^{n(n-t)}/4 < m_{q,t}(n)$$ 
provided $n_0$ is sufficiently large depending on $t$, a contradiction. Therefore, we either have $\J = \langle \Pi \rangle$ where $\Pi \in \Lin(S,V)$ for some $S \leq V$ with $\dim(S)=t$, or we have $\J = \langle \pi \rangle$, where $\pi \in \Lin(A,V^*)$ for some $A \leq V^*$ with $\dim(A)=t$. Suppose first that $\J = \langle \Pi \rangle$ where $\Pi \in \Lin(S,V)$ for some $S \leq V$ with $\dim(S)=t$. We may assume that $V = \mathbb{F}_q^n$, and by a suitable change of basis if necessary, we may assume that $S = \Span\{e_1,\ldots,e_t\}$ and that $\Pi(e_i) = e_i$ for all $i \in [t]$, where $\{e_1,\ldots,e_t\}$ is the standard basis of $\mathbb{F}_q^n$. In other words,
$$\J = \{\sigma \in \Lin(V,V): \sigma(e_i) = e_i\ \forall i \in [t]\}$$
is the family of all linear maps fixing the first $t$ standard basis vectors. Note that $m_{q,t}(n) = |\J \cap \GL(V)|$. Assume for a contradiction that $\f \neq \J \cap \GL(V)$, and let $\tau \in \f\setminus \J$. Then
$$\dim(\{v \in \Span\{e_1,\ldots,e_t\}:\ \tau(v)=v\}) \leq t-1.$$
Let $\mathcal{H}: = \{\sigma \in \J \cap \GL(V):\ \dim(\mathfrak{a}(\sigma,\tau)) = t-1\}$. By Lemma \ref{lem:derangements}, we have $|\mathcal{H}| = \Omega(q^{-(t-1)^2}m_{q,t}(n))$. Since $\f$ is $(t-1)$-intersection-free, we have $\f \cap \mathcal{H} = \emptyset$, and therefore
$$|\f| = |\f \cap \J \cap \GL(V)| + |\f \setminus \J| \leq (1-\Omega(q^{-(t-1)^2}))m_{q,t}(n) + q^{-(t+1)n+O(t^4)}|\Lin(V,V)| < m_{q,t}(n)$$
provided $n_0$ is sufficiently large depending on $t$, a contradiction. Hence, $\f \subset \J \cap \GL(V)$, and therefore $\f = \J \cap \GL(V)$, as required. The case where $\J = \langle \pi \rangle$ follows by considering $\{\sigma^*:\ \sigma \in \f\} \subset \GL(V^*)$, which is also $(t-1)$-intersection-free.
\end{proof}

A very similar argument proves the following.

\begin{theorem}
For any $t \in \mathbb{N}$, there exists $n_0 = n_0(t) \in \mathbb{N}$ such that the following holds. If $n \in \mathbb{N}$ with $n \geq n_0$, $q$ is a prime power, $V$ is an $n$-dimensional vector space over $\mathbb{F}_q$, and $\f \subset \SL(V)$ is $(t-1)$-intersection-free, then
$$|\f| \leq \frac{1}{q-1}\prod_{i=1}^{n-t}(q^n - q^{i+t-1}).$$
Equality holds only if there exists a $t$-dimensional subspace $U$ of $V$ on which all elements of $\f$ agree, or a $t$-dimensional subspace $A$ of $V^*$ on which all elements of $\{\sigma^*:\ \sigma \in \f\}$ agree.
\end{theorem}

\section*{Acknowledgments} 
The authors are grateful to the anonymous reviewers for their careful reading of the paper, and their helpful recommendations.

\bibliographystyle{amsplain}

\begin{thebibliography}{99}
\bibitem{ak} R. Ahlswede and L. H. Khachatrian. \newblock The complete intersection theorem for systems of finite sets. \newblock {\it Europ. J. Combin.}, 18:125--136, 1997.

\bibitem{aa} M. Ahanjideh and N. Ahanjideh. \newblock Erd\H{o}s-Ko-Rado theorem in some linear groups and some projective special linear groups. \newblock {\em Studia Sci. Math. Hungar.} 51:83--91, 2014.

\bibitem{am}  B. Ahmadi and K. Meagher. \newblock The Erd\H{o}s-Ko-Rado property for some permutation groups. \newblock {\em Australas. J. Combin.} 61:23--41, 2015.

\bibitem{cfgs} F.K.R. Chung, P. Frankl, R.L. Graham and J.B. Shearer. \newblock Some intersection theorems for ordered sets and graphs. \newblock {\em J. Combin. Theory, Series A}, 43:23--37, 1986.

\bibitem{DF83} M. Deza and P. Frankl. \newblock Erd\H{o}s-Ko-Rado theorem -- 22 years later. \newblock {\it SIAM J. Alg. Disc. Meth.}, 4:419--431, 1983.

\bibitem{dinur-friedgut} I. Dinur and E. Friedgut. \newblock Intersecting families are essentially contained in juntas. \newblock {\it Combin. Probab. Comput.}, 18:107--122, 2009.

\bibitem{triangle} D. Ellis, Y. Filmus and E. Friedgut. \newblock A proof of the Simonovits-S\'os conjecture. \newblock {\em J. Eur. Math. Soc.}, 14:841--885, 2012.

\bibitem{el} D. Ellis and N. Lifshitz. \newblock Approximation by juntas in the symmetric group, and forbidden intersection problems. \newblock {\em Duke Math. J.}, 171:1417--1467, 2022.

\bibitem{ekl} D. Ellis, N. Keller and N. Lifshitz. \newblock Stability for the complete intersection theorem, and the forbidden intersection problem of Erd\H{o}s and S\'os. \newblock 2016, arXiv:1604.06135.

\bibitem{e-kindler-l} D. Ellis, G. Kindler and N. Lifshitz. \newblock An analogue of Bonami's lemma for functions on spaces of linear maps, and 2-2 games. \newblock {\em Proceedings of the 55th Annual ACM Symposium on Theory of Computing}, June 2023, pp.\ 656--660.

\bibitem{ekr} P. Erd\H{o}s, C. Ko, and R. Rado. \newblock Intersection theorems for systems of finite sets. \newblock {\em
Quart. J. Math. Oxford, Series 2}, 12:313--320, 1961.

\bibitem{erdos-sos} P. Erd{\H{o}}s. \newblock Problems and results in graph theory and combinatorial analysis. \newblock In: Proceedings of the 5th British Combinatorial Conference University of Aberdeen, Aberdeen, 14th-18th July 1975, C. St J. A. Nash-Williams and C. Sheehan (Eds.), Utilitas Mathematica Pub., Winnipeg, 1976.

\bibitem{es} A. Ernst and K.-U. Schmidt. \newblock Intersecting theorems for finite general linear groups. \newblock {\em Math. Proc. Cam. Phil. Soc.}, 175:129--160, 2023.

\bibitem{ff} P. Frankl and Z. F\"uredi. \newblock Forbidding just one intersection. \newblock {\em J. Combin. Theory, Ser. A}, 39:160--176, 1985.

\bibitem{kmw} P. Keevash, D. Mubayi and R. Wilson. \newblock Set systems with no singleton intersection. \newblock {\em SIAM J. Discrete Math.}, 20:1031--1041, 2006.

\bibitem{keller-lifshitz} N. Keller and N. Lifshitz. \newblock The junta method for hypergraphs, and the Erd\H{o}s-Chv\'atal simplex conjecture. \newblock 2017, arXiv:1707.02643.

\bibitem{kz} A. Kupavskii and D. Zakharov. \newblock Spread approximations for forbidden intersection problems. \newblock 2022, arXiv:2203.13379.

\bibitem{meagher} K. Meagher and L. Moura. \newblock Erd\H{o}s-Ko-Rado theorems for uniform set-partition systems. \newblock {\em Electron. J. Comb.}, 12:\#R40, 2005.

\bibitem{morrison} K. E. Morrison. \newblock Integer Sequences and
Matrices Over Finite Fields. \newblock {\em J. Integer Sequences}, 9, Article 06.2.1, 2006.

\bibitem{MV15} D. Mubayi and J. Verstra\"{e}te. \newblock A survey of Tur\'{a}n problems for expansions. \newblock In {\em Recent Trends in Combinatorics}, A. Beveridge, J.R. Griggs, L. Hogben, G. Musiker and P. Tetali (Eds.), IMA Volumes in Mathematics and its Applications 159, Springer, New York, 2015.



\end{thebibliography}


\begin{dajauthors}
\begin{authorinfo}[david]
  David Ellis\\
  School of Mathematics, University of Bristol\\
  Bristol, United Kingdom\\
  david\imagedot{}ellis\imageat{}bristol\imagedot{}ac\imagedot{}uk \\
\end{authorinfo}
\begin{authorinfo}[guy]
  Guy Kindler\\
  The Rachel and Selim Benin School of Computer Science and Engineering, Hebrew University of Jerusalem\\
  Jerusalem, Israel\\
  gkindler\imageat{}cs\imagedot{}huji\imagedot{}ac\imagedot{}il \\
\end{authorinfo}
\begin{authorinfo}[noam]
  Noam Lifshitz\\
  Einstein Insitute of Mathematics, Hebrew University of Jerusalem\\
  Jerusalem, Israel\\
  noam\imagedot{}lifshitz\imageat{}mail\imagedot{}huji\imagedot{}ac\imagedot{}il\\
\end{authorinfo}
\end{dajauthors}

\end{document}